\documentclass[onefignum,onetabnum]{siamart171218}

\usepackage{lipsum}
\usepackage{amsfonts}
\usepackage{graphicx}

\usepackage{epstopdf}
\usepackage{algorithm}
\usepackage{algorithmic}

\usepackage{multirow}
\usepackage{listings}
\usepackage{mathtools}
\usepackage{latexsym,amsmath,amsfonts,amscd}
\usepackage{subfigure}

\usepackage{cases}
\usepackage{bbm,ulem}
\usepackage{amsopn}

\usepackage{graphicx}
\usepackage{amsmath}

\newtheorem{remark}[theorem]{Remark} 

\ifpdf
  \DeclareGraphicsExtensions{.eps,.pdf,.png,.jpg}
\else
  \DeclareGraphicsExtensions{.eps}
\fi

\newcommand{\creflastconjunction}{, and~}

\headers{Spectral element method for wave equations}{H. Li,  D. Appel\"{o}, and X. Zhang} 

\title{Accuracy  of spectral element method for  wave, parabolic and Schr\"{o}dinger equations 
\thanks{H. Li and X. Zhang were supported by the NSF grant DMS-1522593. D. Appel\"{o} was supported in part by NSF Grant DMS-1913076. Any conclusions or recommendations expressed in this paper are those of the authors and do not necessarily reflect the views of the NSF.}
}

\author{Hao Li\thanks{Department of Mathematics,
Purdue University,
150 N. University Street,
West Lafayette, IN 47907-2067
  (\email{li2497@purdue.edu}, \email{zhan1966@purdue.edu}).}
  \and Daniel Appel\"{o}\thanks{Department of Computational Mathematics, Science, and Engineering; Department of Mathematics, Michigan State University, East Lansing, MI 48824
(\email{appeloda@msu.edu}).} \and Xiangxiong Zhang\footnotemark[2]}

\begin{document}
\maketitle

\begin{abstract}
 The spectral element method constructed by the $Q^k$ ($k\geq 2$) continuous finite element method with $(k+1)$-point Gauss-Lobatto quadrature on rectangular meshes
is a popular high order scheme for solving wave equations in various applications. It can also be regarded as 
a finite difference  scheme on all Gauss-Lobatto points. We prove that this finite difference scheme is $(k+2)$-order accurate in discrete 2-norm for smooth solutions.   The same proof can be extended to the spectral element method solving linear parabolic and Schr\"odinger equations.
The main result also applies to the spectral element method on curvilinear meshes that can be smoothly mapped to  rectangular meshes on the unit square. \end{abstract}

\section{Introduction}
Accurate and efficient approximations of solutions to partial differential equations are important to numerous applications arising in engineering and the sciences. In particular for problems whose solutions are of wave type, high order accurate methods are favored as they can control the dispersive errors in wave forms that propagate over vast distances.  

For wave equations and other hyperbolic problems, the two key insights that a numerical analyst can provide to a practitioner comparing methods are: a) if the method is guaranteed to be stable, and b) if the numerical method is guaranteed to be accurate. The first condition is most conveniently guaranteed by selecting a method that is based on a variational formulation such as spectral elements, summation-by-parts 
s and continuous and discontinuous Galerkin finite element methods.  

In recent years many such stable and high order accurate methods  for wave equations have been developed. These include discontinuous Galerkin methods for first order hyperbolic systems 
\cite{HesthavenWarburton02,Monk:2005aa,ChungEngquist06,ChungEngquist09,HesthavenWarburton08,TW_low_storage_curve,MengShuWu15}
and wave equations in second order form \cite{IPDG_Elastic,GSSwave,ChouShuXing2014,Upwind2}, and finite differences with summation by parts operators \cite{Mattsson2012,Mattsson2004,Mattsson2008,Virta2014,Almquist2019,Wang2018,Wang2016}, as well as spectral elements for wave equations \cite{komatitsch1999spectral, komatitsch1999introduction}. 

In this paper we are mainly concerned with the second topic, to provide rigorous estimates on the errors for a method. In particular, we study the rates of convergence of the error,  as measured in norms over nodes for all degree of freedoms, for the spectral element method applied to linear wave and parabolic, and   Schr\"{o}dinger equations. { These three types of equations are fundamentally different, but all of them contain the same second order operator, which can be discretized by the same spectral element method. }

To be precise, we consider the  Lagrangian $Q^k$ ($k\geq 2$) continuous finite element method for solving linear evolution PDEs with a second order operator $\nabla\cdot (\mathbf a(\mathbf x)\nabla u)$ on rectangular meshes implemented by $(k+1)$-point Gauss-Lobatto quadrature for all integrals. This is often referred to as the spectral element method in the literature and this is the notation we will use here. 

For the $Q^k$ spectral element method, it is well known that the standard finite element error estimates  still hold \cite{maday1990optimal}, 
i.e., the error in $H^1$-norm is $k$-th order and the error in $L^2$-norm is $(k+1)$-th order. It is also well known that the Lagrangian $Q^k$ ($k\geq 2$) continuous finite element method is $(k+2)$-th order accurate in the discrete 2-norm over all $(k+1)$-point Gauss-Lobatto quadrature points \cite{wahlbin2006superconvergence, lin1996, chen2001structure}. If using a very accurate quadrature in the finite element method for a variable coefficient operator $\nabla\cdot (\mathbf a(\mathbf x)\nabla u)$, then  $(k+2)$-th order superconvergence at Gauss-Lobatto points holds trivially. However, for the efficiency of having a diagonal mass matrix and for the convenience of implementation,  the most popular method for wave equations is the simplest choice of quadrature, i.e. using $(k+1)$-point Gauss-Lobatto quadrature for $Q^k$ elements in all integrals for both mass and stiffness matrices. In particular in the seismic community, where highly efficient simulation of the elastic wave equation is of important, the spectral method has become the method of choice,  \cite{komatitsch1999spectral, komatitsch1999introduction}.

When using this $(k+1)$-point Gauss-Lobatto quadrature for Lagrangian  $Q^k$  finite element method,  the quadrature nodes coincide with the nodes defining the degrees of freedom,  and the resulting method becomes the so-called spectral element method.  Thus the spectral element method can also be regarded as a finite difference   scheme at all Gauss-Lobatto points. For instance, consider solving $u_{tt}=u_{xx}$ on the interval $[0,1]$ with homogeneous Dirichlet boundary conditions.  Introduce the  uniform grid $0=x_0<x_1<\cdots<x_N<x_{N+1}=1$ with spacing $h = 1/(N+1)$ and $N$ being odd. This grid  gives a uniform partition of the interval $[0,1]$ into uniform intervals $I_k=[x_{2k}, x_{2k+2}]$ $(k=0,\cdots, \frac{N-1}{2})$.  
Then all 3-point Gauss-Lobatto  quadrature points for  intervals $I_k=[x_{2k}, x_{2k+2}]$   coincide with  the grid points $x_i$. 
The  $Q^2$ spectral element method on  intervals $I_k=[x_{2k}, x_{2k+2}]$ $(k=0,\cdots, \frac{N-1}{2})$ is equivalent to the following semi-discrete finite difference scheme \cite{cohen2001higher,li2020superconvergence}:
\begin{subequations}
\label{1d-p2}
\begin{align}
\frac{d^2}{dt^2}u_i &=\frac{u_{i-1}-2u_i+u_{i+1}}{h^2},\quad \mbox{if $i$ is odd;}\\
\frac{d^2}{dt^2}u_i &=\frac{-u_{i-2}+8u_{i-1}-14u_i+8u_{i+1}-u_{i+2}}{4h^2},\quad \mbox{if $i$ is even.}
\end{align}
\end{subequations}
While the truncation error of \eqref{1d-p2} is only second order yet the dispersion error is fourth order, see Section 11 in \cite{cohen2001higher}. Although the dispersion error results can in principle be extended to any order, the derivation and expressions become increasingly cumbersome. Further the dispersion error results are limited to unbounded or periodic domains and do not produce error estimates in the form of a norm of the error.  
{ Other than spectral element methods,
other high order schemes can also be interpreted as a finite difference scheme, such as 
the Fourier pseudo-spectral method \cite{CiCP-26-1335, gottlieb2012stability, cheng2015fourier}.}

In fact, as we have shown in \cite{li2020superconvergence},  it is nontrivial and requires new analysis tools to establish the $(k+2)$-th order superconvergence when $(k+1)$-point Gauss-Lobatto quadrature is used.   In \cite{li2020superconvergence},  $(k+2)$-th order accuracy at all Gauss-Lobatto points of $Q^k$ spectral element method was proven for elliptic equations with Dirichlet boundary conditions. In this paper, we  extend those results and  will prove  that the $Q^k$ spectral element method is a  $(k+2)$-th order accurate  scheme for linear wave, parabolic and  Schr\"{o}dinger equations  with Dirichlet boundary conditions. For Neumann boundary conditions, if $\mathbf a(\mathbf x)$ is diagonal, i.e., there are no mixed second order derivatives in $\nabla\cdot(\mathbf a(\mathbf x) \nabla u)$, 
$(k+2)$-th order accuracy in discrete 2-norm can be proven. When mixed second order derivatives are involved,
only $(k+\frac32)$-th order can be proven for Neumann boundary conditions, and we indeed observe some order loss in numerical tests.
 
 The main contribution of this paper is to explain the order of accuracy of $Q^k$ spectral element method, when the errors are measured only at nodes of degree of freedoms. As mentioned above we consider the case of rectangular elements and a smooth coefficient $\mathbf a(\mathbf x)$ in the term $\nabla\cdot (\mathbf a(\mathbf x)\nabla u)$.  We note that this does include discretizations on regular meshes of curvilinear domains that can be smoothly mapped to rectangular meshes for the unit cube,  e.g., the spectral element method for $\Delta u$ on such a mesh for a curvilinear domain
 is equivalent to the spectral element method for $\nabla\cdot (\mathbf a(\mathbf x)\nabla u)+\mathbf b(\mathbf x)\cdot \nabla u$ on a reference uniform rectangular mesh where $\mathbf a(\mathbf x)$ and $\mathbf b(\mathbf x)$ emerge from the mapping between the curvilinear domain and the unit cube.  It does however not include problems on unstructured quadrilateral meshes where the metric terms typically are non-smooth at element interfaces but we note that the numerical examples that we present indicate that such meshes may still exhibit larger rates than $k+1$.  {  We only consider the semi-discrete schemes for linear equations in this paper. In general, it is straightforward to extend the error estimates to a fully discrete scheme for simple time discretizations, e.g., \cite{wheeler1973priori}. Even though superconvergence in $Q^k$ finite element method without any quadrature can be established for nonlinear equations \cite{chen2001structure}, the result in this paper may no longer hold for generic nonlinear equations since the simplest $(k+1)$-point Gauss-Lobatto quadrature are not accurate enough for nonlinear terms. }
 
This paper is organized as follows. In Section \ref{sec-preliminaries}, we introduce notation and assumptions. In Section \ref{lemma-quadrature-error-estimate}, we review a few standard quadrature estimates. In Section \ref{elliptic-proj-error-estimate}, the superconvergence of elliptic projection is analyzed, which is parallel to the classic error estimation for hyperbolic and parabolic equations by involving elliptic projection of the corresponding elliptic operator, see \cite{wheeler1973priori, sammon1982convergence, dupont19732}. We then prove the main result for homogeneous Dirichlet boundary conditions in Section \ref{error-estimate}, for  the second-order wave equation in Section \ref{error-wave},  parabolic equations   in Section \ref{error-heat}  and linear Schr\"{o}dinger equation in Section \ref{error-schrodinger}. Neumann boundary conditions can be discussed similarly as summarized in Section \ref{sec-Neumann}. For problems with nonhomogeneous Dirichlet boundary conditions, a convenient implementation which maintains the $(k+2)$-th order of accuracy is given in Section \ref{implementation-nonhom}. 
 Numerical tests verifying the estimates are given in Section \ref{numerical-test}.
  Concluding remarks are given in Section \ref{concludingremark}
\section{Equations, notation, and assumptions}
\label{sec-preliminaries}

\subsection{Problem setup}
Let $L$ be a linear  second order differential operator with time dependent coefficients:
\begin{align*}
Lu = -\nabla\cdot(\mathbf a(\mathbf x,t)\nabla u) + \mathbf b(\mathbf x,t)\cdot \nabla u + c(\mathbf x,t)u,
\end{align*} 
 where $\mathbf a(\mathbf x,t)=(a_{ij}(\mathbf x,t))$ is a positive symmetric definite operator for $t\in [0,T]$, i.e. there exist constants $\alpha,\beta >0$ such that $\alpha |\xi|^2\leq \xi^T \mathbf a(\mathbf x,t) \xi \leq \beta |\xi|^2$,
for all $(\mathbf x,t)\in \Omega\times[0,T], \xi \in \mathbbm R^n.$ 
Consider the following two initial-boundary value problems with smooth enough coefficients on 
a rectangular domain  $\Omega=(0,1)\times(0,1)$ with its boundary  $\partial \Omega$: 

Given $0<T<\infty$, find $u( \mathbf x,t)$ on $\bar \Omega \times [0,T]$ satisfying
\begin{equation}\label{heat-equation-homo-bc}
\begin{aligned}
u_t
 = & -Lu + f(\mathbf x,t) & \text{ in } \Omega \times (0,T], \\
u(\mathbf x,t)  = & 0 &\text{ on } \partial\Omega \times [0,T], \\
u(\mathbf x,0) = & u_0(\mathbf x) &\text{ on } \Omega.
\end{aligned}
\end{equation}

Given $0<T<\infty$, find $u(\mathbf x,t)$ on $\bar \Omega \times [0,T]$ satisfying  
\begin{equation}\label{wave-equation-homo-bc}
\begin{aligned}
u_{tt}
 = & -Lu + f(\mathbf x,t) & \text{ in } \Omega \times (0,T], \\
u(\mathbf x,t)  = & 0 &\text{ on } \partial\Omega \times [0,T], \\
u(\mathbf x,0) = & u_0(\mathbf x) ,\quad u_t(\mathbf x,0) = u_1(\mathbf x) &\text{ on } \Omega\times\{t=0\}.
\end{aligned}
\end{equation}
  We use $A(\cdot)$ to denote the   bilinear form: for $u,v \in H^1(\Omega)$,
\begin{align}\label{nonsymmetric-bilinearform}
A(u,v) = \int_{\Omega}\nabla u^T \mathbf a(\mathbf x,t)\nabla v+ \mathbf b(\mathbf x,t)\cdot \nabla u + c(\mathbf x,t)uv \, d\mathbf x. 
\end{align}
 
For convenience, we assume $\Omega_h$ is a uniform rectangular mesh for $\bar\Omega$ and $e=[x_e-h,x_e+h]\times [y_e-h,y_e+h]$ denotes any cell in $\Omega_h$ with  cell center $(x_e,y_e)$.  Though we only discuss uniform meshes, the main result can be easily extended to nonuniform rectangular meshes with smoothly varying cells.   Let $\mbox{\scriptsize $Q^k(e)=\left\{p(x,y)=\sum\limits_{i=0}^k\sum\limits_{j=0}^k p_{ij} x^iy^j,   (x,y)\in e\right\}$ }$,  denote  the set of 
 tensor product of polynomials of degree $k$ on  an element $e$. Then we use $V^h=\{p(x,y)\in C^0(\Omega_h): p|_e \in Q^{k}(e),\quad \forall e\in \Omega_h\}$ to denote the continuous piecewise $Q^{k}$ finite element space on $\Omega_h$ and $V^h_0=\{v_h\in V^h: v_h|_{\partial \Omega}=0 \}.$  Let $(u,v)=\int_\Omega u v d\mathbf x$  and let  $\langle \cdot, \cdot \rangle_h$ and $A_h(\cdot, \cdot)$ denote approximation of the integrals by $(k+1)$-point Gauss-Lobatto quadrature for each spatial variable in each cell. Also,  $u^{(i)}$ will denote the $i$-th time derivative of the function $u(\mathbf x,t)$.
 
 For  the   equations that we are interested in, assume the exact solution $u(\mathbf x,t)\in H_0^1(\Omega)\cap H^2(\Omega)$  for any $t$, and define its discrete elliptic projection $R_hu \in V^h_0$ as
\begin{equation}\label{elliptic-proj}
A_h(R_h u,v_h) = \langle -Lu,v_h \rangle_h,\quad \forall v_h\in V^h_0, \quad 0\leq t \leq T.
\end{equation}
 Also, let  $u_I \in V^h$ denote the piecewise Lagrangian $Q^k$ interpolation polynomial of function $u$ at $(k+1)\times(k+1)$ Gauss-Lobatto points in each rectangular cell.
 
We consider  semi-discrete spectral element schemes whose initial conditions are defined by the elliptic projection and the Lagrange interpolant of the continuous initial data. 

For problem \eqref{heat-equation-homo-bc} the scheme is 
to find $u_h(\mathbf x, t) \in V^h_0$ satisfying
\begin{equation}\label{heat-scheme}
\begin{aligned}
\langle u_h^{(1)}, v_h \rangle_h + A_h(u_h, v_h) = & \langle f, v_h \rangle_h,\quad  \forall v_h \in V^h_0, \\
u_h(0) = & R_hu_0.
\end{aligned}
\end{equation} 

We consider the semi-discrete  spectral element scheme for problem \eqref{wave-equation-homo-bc} with special initial conditions: solve for $u_h( t) \in V^h_0$ satisfying
\begin{equation}\label{wave-scheme}
\begin{aligned}
\langle u_h^{(2)}, v_h \rangle_h + A_h(u_h, v_h) = & \langle f, v_h \rangle_h,\quad  \forall v_h \in V^h_0,\\
u_h(0) = R_hu_0, \quad u_h^{(1)}(0) = & (u_1)_I.
\end{aligned}
\end{equation}
 
\subsection{Notation and basic tools}
We will use the same  notation  as in \cite{li2020superconvergence1, li2020superconvergence}. 

The norm and semi-norms for $W^{k,p}(\Omega)$ and $1\leq p<+\infty$, with standard modification for $p=+\infty$ can be defined as follows,
$$
 \|u\|_{k,p,\Omega}=\left(\sum\limits_{i+j\leq  k}\iint_{\Omega}|\partial_x^i\partial_y^ju(x,y)|^pdxdy\right)^{1/p},
 $$
$$
 |u|_{k,p,\Omega}=\left(\sum\limits_{i+j= k}\iint_{\Omega}|\partial_x^i\partial_y^ju(x,y)|^pdxdy\right)^{1/p}.
 $$
 When there is no confusion, for simplicity, sometimes we may use $\|u\|_{k}$ and $|u|_{k}$ as norm and semi-norm for $H^k(\Omega)=W^{k,2}(\Omega)$ respectively.

 For any $v_h\in V^h$, $1\leq p<+\infty $, and $k\geq 1$,  we define the broken broken Sobolev norms and seminorms by the following symbols,
  $$\|v_h\|_{k,p,\Omega}:= \left(\sum_e\|v_h\|_{k,p,e}^p\right)^{\frac1p}, \quad 
 |v_h|_{k,p,\Omega}:= \left(\sum_e|v_h|_{k,p,e}^p\right)^{\frac1p}.$$
 
  Let $Z_{0,e}$ denote the set of $(k+1)\times(k+1)$ Gauss-Lobatto points of the cell $e$ and $Z_0=\bigcup_e Z_{0,e}$ denote all Gauss-Lobatto points in the mesh $\Omega_h$. Let $\|u\|_{l^2(\Omega)}$ and $\|u\|_{l^{\infty}(\Omega)}$
denote the discrete 2-norm and the maximum norm over $Z_0$ respectively as 
\[\|u\|_{l^2(\Omega)}=\left[h^2\sum_{(x,y)\in Z_0} |u(x,y)|^2\right]^{\frac12},\quad \|u\|_{l^{\infty}(\Omega)}=\max_{(x,y)\in Z_0} |u(x,y)|.\]
 
  When there is no confusion, for simplicity, sometimes we may use $\|u\|_{l^2}$ and $|u|_{l^{\infty}}$ to denote $\|u\|_{l^2(\Omega)}$ and $\|u\|_{l^{\infty}(\Omega)}$ respectively.
  For a continuous function $f(x,y)$, let $f_I(x,y)$ denote its piecewise $Q^k$ Lagrange interpolant at $Z_{0,e}$ on each cell $e$, i.e., $f_I\in V^h$ satisfies:
\[f(x,y)=f_I(x,y), \quad \forall (x,y)\in Z_0. \]
 
 Let $(f,v)_e$ denote the inner product in  $L^2(e)$ and $(f,v)$ denotes the inner product in $L^2(\Omega)$ as
\[(f,v)_e=\iint_{e} fv\, dxdy,\quad (f,v)=\iint_{\Omega} fv\, dxdy=\sum_e (f,v)_e.\]
 Let  $\langle f,v\rangle_h$ denote the approximation to $(f,v)$ by using $(k+1)\times(k+1)$-point Gauss-Lobatto quadrature for integration over each cell $e$.  Then for $k\geq 2$, the $(k+1)\times(k+1)$ Gauss-Lobatto quadrature is exact for 
integration of  tensor product   polynomials of degree $2k-1\geq k+1$ on $\hat K$.

We denote $A^{*}(\cdot, \cdot)$ as the adjoint bilinear form of $A(\cdot, \cdot)$ such that
\[
A^{*}(v, u)=A(u, v)=(\mathbf{a} \nabla u, \nabla v)+(\mathbf{b} \cdot \nabla u, v)+(c u, v).
\]

Let superscript $(i)$ denote $i$-th time derivatives for coefficients $\mathbf a, \mathbf b$, and $c$. 
For the time dependent operators $L$ and $A$, the symbols $L^{(i)}$ and $A^{(i)}$ are defined as taking time derivatives only for coefficients:
\[ 
L^{(i)}u = -\nabla\cdot(\mathbf a^{(i)}\nabla u) + \mathbf b^{(i)}\cdot \nabla u + c^{(i)}u,\]
and
\[A^{(i)}(u,v) = \int_{\Omega}\nabla u^T \mathbf a^{(i)}\nabla v+ \mathbf b^{(i)}\cdot \nabla u + c^{(i)}uv d\mathbf x. \]
The symbol $A^{(i)}_h$ is similarly defined as taking time derivatives only for coefficients in $A_h$. 
With this notation, for $u(\mathbf x,t)$ and time independent test function $v(\mathbf x)$, we have Leibniz rule 
$$(L u)^{(m)}=\sum_{j=0}^m \binom{m}{j} L^{(m-j)} u^{(j)},\quad \left[A(u,v)\right]^{(m)}= \sum_{j=0}^m \binom{m}{j} A^{(m-j)} (u^{(j)}, v).$$

 By integration by parts, it is straightforward to verify
\begin{equation}
\label{notation-formula-1}
(L ^{(m-j)} u^{(j)},v)= A^{(m-j)} (u^{(j)}, v),\quad\forall  v\in H_0^1(\Omega).
\end{equation}

 There exist constants $C_i$ ($i=1,2,3,4$) independent of $h$ such that $l^2$-norm and $L^2$-norm are equivalent for $V^h$:
\begin{equation}\label{norm-equivalence}
\begin{aligned}
C_1\|v_h\|_{l^2} \leq \|v_h\|_0 \leq C_2\|v_h\|_{l^2}, \quad  \forall v \in V^h,\\
C_3\langle v_h, v_h\rangle_h \leq \|v_h\|^2_0 \leq C_4 \langle v_h, v_h\rangle_h, \quad \forall v \in V^h.
\end{aligned}
\end{equation}
We have the inverse inequality for polynomials as
 \begin{equation}
 \|v_h\|_{k+1, e}\leq C h^{-1} \|v_h\|_{k, e},\quad  \forall v_h \in V^h,\, k\geq 0. 
\label{inverseestimate}
 \end{equation}

%
\subsection{Assumption on the coercivity and the elliptic regularity}

For the operator $
A(u,v): =\int_\Omega [\nabla u^T \mathbf a \nabla v +(\mathbf b\cdot \nabla u) v + c u v] \,d\mathbf x $ 
where $\mathbf a=\begin{pmatrix}
               a^{11} & a^{12}\\
               a^{21} & a^{22}
              \end{pmatrix}
$ is positive definite and $\mathbf b=[b^1 \quad b^2]$, 
assume the coefficients $a_{ij}$, $b_j$, $c \in C^{m_1}\left([0,T];W^{m_2,\infty}(\Omega)\right)$ for $m_1$, $m_2$ large enough. Thus for $t\in [0,T]$, $A(u,v)\leq C\|u\|_1\|v\|_1$ for any $u, v\in H^1_0(\Omega)$. As discussed in \cite{li2020superconvergence}, if we assume $\lambda_{\mathbf a}$ has a positive lower bound and $ \nabla \cdot \mathbf b\leq 2c $, where $\lambda_{\mathbf a}$ as the smallest eigenvalues of $\mathbf a$, the coercivity of the bilinear form can be easily achieved. 
For the $V^h$-ellipticity, as pointed out in Lemma 5.2 of \cite{li2020superconvergence}, 
if $4\lambda_{\mathbf a}c > |\mathbf b|^2$, for $t\in[0,T]$,
\begin{equation}\label{vh-ellipticity}
C\|v_h\|_1^2 \leq A_h(v_h,v_h), \quad \forall v_h \in V^h,
\end{equation}
can be proven. In the rest of this paper, we assume coercivity for the bilinear forms $A$, $A^*$, and $A_h$. We  assume the elliptic regularity $\|w\|_2\leq C\|f\|_0$ holds for the exact dual problem of finding $w\in H^1_0(\Omega)$ satisfying
$A^*(w,v)=(f,v),\quad \forall v\in H_0^1(\Omega)$.
 See \cite{savare1998regularity, grisvard2011elliptic} for the elliptic regularity with Lipschitz continuous coefficients on a Lipschitz domain. 

We remark that in the case of the wave equation we also assume finite speed of propagation i.e. that there is an upper bound on the eigenvalues of ${\bf a}$.  

\section{Quadrature error estimates}\label{lemma-quadrature-error-estimate} 
 
For any continuous function $u(\mathbf x,t_0)$ with fixed time $t_0$,
its M-type projection on spatial variables  is a continuous piecewise $Q^k$ polynomial of $\mathbf x$,  denoted as $u_p(\mathbf x,t_0)\in V^h$. The M-type projection was used to analyze superconvergence \cite{chen2001structure}. 
Detailed definition and some useful properties about the M-type projection can be also found in \cite{li2020superconvergence1, li2020superconvergence}. For $m\geq 0$, $\left(u_p\right)^{(m)} = \left(u^{(m)}\right)_p$, thus there is no ambiguity to use the notation $u_p^{(m)}$. 
The M-type projection has the following properties. See Theorem 3.2 in  \cite{li2020superconvergence1} for the detailed proof.

\begin{theorem}
\label{thm-superapproximation}
For $k \geq 2$, 
 \[\|u-u_p\|_{l^2(\Omega)}=\mathcal O(h^{k+2}) \|u\|_{k+2},\quad\forall u\in H^{k+2}(\Omega).\]
 \[\|u-u_p\|_{l^{\infty}(\Omega)}=\mathcal O(h^{k+2}) \|u\|_{k+2,\infty},
 \quad\forall u\in W^{k+2,\infty}(\Omega).\]
\end{theorem}

By applying Bramble-Hilbert Lemma, we have the following standard quadrature estimates. See \cite{li2020superconvergence1} for the detailed proof. 
\begin{lemma}
\label{rhs-estimate}
 For $f(\mathbf x)$, if $f(\mathbf x) \in H^{k+2}(\Omega)$, then we have $$(f,v_h)-\langle f,v_h\rangle_h =\mathcal O(h^{k+2}) \|f\|_{k+2} \|v_h\|_2,\quad \forall v_h\in V^h.$$
\end{lemma}

The next lemma shows the superconvergence of the bilinear form with Gauss-Lobatto quadrature $A_h$, and it collects the results of Lemma 4.5 - Lemma 4.8 of \cite{li2020superconvergence}.
\begin{lemma}\label{bilinear-u-up}
For $i,j \geq 0$ and any fixed $t\in[0,T]$, assuming sufficiently smooth coefficients $\mathbf a,\mathbf b, c$ and function $u(\mathbf x,t)\in H^{(k+3)}(\Omega)$, we have
\begin{align}
\label{bilinear-form-estimate}
A^{(i)}_h((u-u_p)^{(j)},v_h)=\begin{cases}
                         \mathcal O(h^{k+2})\|u^{(j)}(t)\|_{k+3}\|v_h\|_2,\quad \text{if } v_h\in V^h_0 \text{ or $\mathbf a$ is diagonal;}  \\
                         \mathcal O(h^{k+\frac{3}{2}})\|u^{(j)}(t)\|_{k+3}\|v_h\|_2,\quad \text{ otherwise.}
                        \end{cases}
\end{align}
\end{lemma}

The following results are Lemma 3.5, Theorem 3.6, Theorem 3.7 in \cite{li2020superconvergence}.

\begin{lemma}
\label{lemma-quaderror-2norm}
 If $f\in H^{2}(\Omega)$ or $f\in V^h$, we have $$(f,v_h)-\langle f,v_h\rangle_h =\mathcal O(h^2) |f|_{2} \|v_h\|_0,\quad\forall v_h\in V^h.$$
\end{lemma}

\begin{lemma}\label{a-ah-1norm}
Assume all coefficients of \eqref{nonsymmetric-bilinearform} are in $L^{\infty}\left([0,T];W^{2,\infty}(\Omega)\right)$. We have 
\[ A(z_h,v_h)-A_h(z_h,v_h)=\mathcal O(h) \|v_h\|_2 \|z_h\|_1, \quad\forall v_h,z_h\in V^h.\]
\end{lemma}

\begin{lemma}
\label{bilinear-quadrature-error}
For the differential operator $L$ and any fixed $t\in[0,T]$, assume $a_{ij}(\mathbf x,t)$, $b_{i}(\mathbf x,t)$, $c(\mathbf x,t)\in L^{\infty}\left([0,T];W^{k+2,\infty}(\Omega)\right)$ and $u(\mathbf x, t)\in H^{k+3}(\Omega)$. For $k \geq 2$, we have 
\begin{align}
\label{xz-bilinear-quad-estimate}
\resizebox{\textwidth}{!}
     {$
A(u, v_h)-A_h(u, v_h)=\begin{cases}
                         \mathcal O(h^{k+2})\|u(t)\|_{k+3}\|v_h\|_2,\quad \text{if } v_h\in V^h_0 \,\, \mbox{or} \,\, {(\mathbf a\nabla u)\cdot\mathbf n =  0 \,\, \text{on} \,\, \partial\Omega} \\
                         \mathcal O(h^{k+\frac{3}{2}})\|u(t)\|_{k+3}\|v_h\|_2,\quad \text{ otherwise}
                        \end{cases},
                        $}
\end{align}
{ where $\mathbf n$ denotes the unit vector normal to the domain boundary $\partial\Omega$.  }
\end{lemma}

\begin{remark}
\label{rmk-loss}
There is half order loss in \eqref{bilinear-form-estimate}, only when using $v \in V^h$ for non-diagonal $\mathbf a$, i.e., when solving second order equations containing mixed
second order derivatives with homogeneous Neumann boundary conditions.  {\color{blue} See \cite{Li2021} for detailed proof of \eqref{xz-bilinear-quad-estimate} for the  homogeneous Neumann boundary condition case, i.e., $(\mathbf a\nabla u)\cdot\mathbf n =  0$ along the domain boundary. }
 \end{remark}

We have the Gronwall's inequality in integral form as follows:
\begin{lemma}\label{gronwall-inequality}
Let $\xi(t)$ be continuous on $[0,T]$ and
\begin{equation*}
\xi(t)\leq C_1 \int_0^t \xi(s)ds + \alpha(t)
\end{equation*}
for constant $C_1\geq 0$ and $\alpha(t)\geq 0$ nondescreasing in $t$. Then $ \xi(t) \leq \alpha(t) e^{C_1 t}$  thus $ \xi(t) \leq \alpha(t) e^{C_1 T}=C\alpha(t)$ for all $0\leq t \leq T$.
\end{lemma}

\section{Error estimates for the elliptic projection}\label{elliptic-proj-error-estimate}
Let  $u_h(\mathbf x, t)$ denote the solution of the semi-discrete numerical scheme. 
Let $e(\mathbf x, t) = u_h(\mathbf x, t) - u_p(\mathbf x, t)$, then we can write 
$$e=\theta_h + \rho_h,$$ 
where $\theta_h := u_h-R_h u\in V^h_0$ and $\rho_h := R_hu - u_p \in V^h_0$. 

In this section, we will establish the superconvergence result for the elliptic projection, which is an important step for proving the superconvergence of function values. 
We have the following superconvergence result for $\|\rho_h^{(m)}(t)\|$, $m\geq 0$, $t\in[0,T]$. 
\begin{lemma}\label{rho-m-estimate}
If $a_{ij}$, $b_j$, $c\in C^m\left([0,T];W^{k+2,\infty}(\Omega)\right)$, $u\in C^{m}\left([0,T];H^{k+4}(\Omega)\right)$, then we have
\begin{align}
\|\rho_h^{(m)}( t)\|_{1} \leq & C h^{k+1}\sum_{j=0}^m(\|u^{(j)}( t)\|_{k+3}+\|(Lu)^{(j)}( t)\|_{k+2}),\label{rho-1-estimate}\\
\|\rho_h^{(m)}\|_{L^2([0,T];L^2(\Omega))} \leq & C h^{k+2}\sum_{j=0}^m(\|u^{(j)}\|_{L^2([0,T];H^{k+3}(\Omega))}+\|(Lu)^{(j)}\|_{L^2([0,T];H^{k+2}(\Omega))}),\label{rho-m-l2-estimate}
\end{align}
\begin{equation}\label{rho-m-infty-estimate}
    \resizebox{\textwidth}{!}
     {%
        $\|\rho_h^{(m)} \|_{L^\infty([0,T];L^2(\Omega))} \leq C h^{k+2}\sum_{j=0}^m(\|u^{(j)}\|_{L^{\infty}([0,T];H^{k+3}(\Omega))}+\|(Lu)^{(j)}\|_{L^{\infty}([0,T];H^{k+2}(\Omega))})$%
     }
\end{equation}
where $C$ is independent of $h$, $u$, $f$, and time $t$.
\end{lemma}
\begin{proof}
First we prove \eqref{rho-1-estimate}, with which  we then  prove \eqref{rho-m-l2-estimate} and \eqref{rho-m-infty-estimate} by the dual argument.

From the definition of  the  discrete elliptic projection \eqref{elliptic-proj} we have
\begin{align}\label{rho-bilinearform-ep}
A_h(\rho_h, v_h) = \epsilon(v_h), \quad \forall v_h \in V^h_0.
\end{align}
where 
\[\epsilon(v_h) = \langle  - Lu, v_h \rangle_h - A_h(u_p,v_h).\]
Note that $v_h$ is time independent. 
Taking $m$ time derivatives of \eqref{rho-bilinearform-ep} yields 
\begin{equation}\label{dt-A-rho}
\left(A_h(\rho_h, v_h)\right)^{(m)} =\sum_{j=0}^m {m\choose j} A_h^{(m-j)}( \rho^{(j)}_h,v_h)= \epsilon^{(m)}(v_h).
\end{equation}
The term $\epsilon^{(m)}(v_h)$ can be rewritten as follows:
\begin{equation*}
\begin{aligned}
& \epsilon^{(m)}(v_h) = \langle (Lu)^{(m)} ,v_h \rangle_h - (A_{h}(u_p,v_h))^{(m)}  \\
= & \left[((Lu)^{(m)} ,v_h ) - (A(u,v_h))^{(m)}\right] -\left[ ((Lu)^{(m)} ,v_h )- \langle (Lu)^{(m)} ,v_h \rangle_h\right] \\ 
&+ \left[(A(u,v_h))^{(m)}-(A_h(u,v_h))^{(m)}\right] + \left(A_{h}(u-u_p,v_h)\right)^{(m)}. 
\end{aligned}
\end{equation*}
By Leibniz rule and \eqref{notation-formula-1}, we have 
\[ ((Lu)^{(m)} ,v_h ) - (A(u,v_h))^{(m)}=\sum_{j=0}^m {m\choose j} \left[(L^{(m-j)} u^{(j)}, v_h)-A^{(m-j)}( u^{(j)},v_h)\right] =0.\]
By Lemma \ref{rhs-estimate}, 
\[((Lu)^{(m)} ,v_h )- \langle (Lu)^{(m)} ,v_h \rangle_h=\mathcal O(h^{k+2})\|(Lu)^{(m)}( t)\|_{k+2} \|v_h\|_2.\]
By Leibniz rule and  Lemma \ref{bilinear-quadrature-error},
\begin{align*} (A(u,v_h))^{(m)}-(A_h(u,v_h))^{(m)}&= \sum_{j=0}^m {m\choose j} \left[A^{(m-j)}(u^{(j)},v_h)-A_h^{(m-j)}(u^{(j)},v_h)\right]\\
&=\mathcal O(h^{k+2}) \sum_{j=0}^m {m\choose j} \|u^{(j)}( t)\|_{k+3}\|v_h\|_2. \end{align*}
 Now,  Lemma \ref{bilinear-u-up} implies
\begin{align*}
\left(A_{h}(u-u_p,v_h)\right)^{(m)}= & \sum_{j=0}^m {m\choose j} A_h^{(m-j)}\left( (u-u_p)^{(j)},v_h\right)\\
= & \mathcal O(h^{k+2}) \sum_{j=0}^m {m\choose j} \|u^{(j)}( t)\|_{k+3}\|v_h\|_2. 
\end{align*}
Thus we have
\begin{equation}
\label{epsilon-m-estimate}
\epsilon^{(m)}(v_h) =  \mathcal O(h^{k+2})\left(\sum_{j=0}^m \|u^{(j)}( t)\|_{k+3}+\|(Lu)^{(m)}( t)\|_{k+2}\right)\|v_h\|_2.
\end{equation}

For $i\geq 0$, by the $V_h$-ellipticity \eqref{vh-ellipticity}, \eqref{dt-A-rho}, and \eqref{epsilon-m-estimate} we have 
\[\begin{aligned}
& C\|\rho_h^{(i)}( t)\|_1^2 \leq A_h( \rho^{(i)}_h, \rho^{(i)}_h)\\
= & \sum_{j=0}^i {i\choose j} A_h^{(i-j)}( \rho^{(j)}_h, \rho^{(i)}_h) - \sum_{j=0}^{i-1} {i\choose j} A_h^{(i-j)}( \rho^{(j)}_h, \rho^{(i)}_h)\\
= & \epsilon^{(i)}(\rho^{(i)}_h) - \sum_{j=0}^{i-1} {i\choose j} A_h^{(i-j)}( \rho^{(j)}_h, \rho^{(i)}_h) \\
\leq & \mathcal O(h^{k+1})\left( \sum_{j=0}^i \|u^{(j)}\|_{k+3}+\|(Lu)^{(i)}\|_{k+2} \right)h\|\rho^{(i)}_h\|_2 +C\sum_{j=0}^{i-1}\|\rho^{(j)}_h( t)\|_{1}\|\rho^{(i)}_h( t)\|_{1}\\
\leq & \left[ \mathcal O(h^{k+1})\left(\sum_{j=0}^i\|u^{(j)}\|_{k+3}+\|(Lu)^{(i)}\|_{k+2}\right)+C\sum_{j=0}^{i-1}\|\rho^{(j)}_h( t)\|_{1}\right]\|\rho^{(i)}_h( t)\|_1,
\end{aligned}\]
 the last inequality follows from an application of an inverse estimate. Thus
\begin{equation}
\label{rho-i-1-norm-estimation}
\|\rho_h^{(i)}( t)\|_1\leq \mathcal O(h^{k+1})\left(\sum_{j=0}^i\|u^{(j)}\|_{k+3}+\|(Lu)^{(i)}\|_{k+2}\right)+C\sum_{j=0}^{i-1}\|\rho^{(j)}_h( t)\|_{1}.
\end{equation}

Now \eqref{rho-1-estimate} can be proven by induction as follows. First, set $i=0$ in  \eqref{rho-i-1-norm-estimation}  to obtain  \eqref{rho-1-estimate} with $m=0$. 
Second, assume   \eqref{rho-i-1-norm-estimation}  holds for $m=i-1$, then  \eqref{rho-i-1-norm-estimation} implies  that \eqref{rho-1-estimate} also holds for $m=i$.
 
For fixed $t\in[0,T]$, to estimate $\rho_h^{(m)}$ in $L^2$-norm, we consider the dual problem: find $\phi_h\in V^h_0$ satisfying: for $i\geq 0$,
\begin{equation}\label{dual-problem}
A^*(\phi_h, v_h)=( \rho^{(i)}_h(t), v_h), \quad \forall v_h \in V^h_0.
\end{equation}
Based on Theorem 5.3 in \cite{li2020superconvergence}, by assuming the elliptic regularity and $V^h$ ellipticity, problem \eqref{dual-problem} has a unique solution satisfying
\begin{equation}\label{dual-elliptic-regularity}
\|\phi_h\|_2 \leq C \|\rho_h^{(i)}(t)\|_{0}.
\end{equation}
Take $v_h = \rho^{(i)}_h$ in \eqref{dual-problem} then we have 
\begin{align*}
& \|\rho^{(i)}_h( t)\|_0^2\\
= & A^*(\phi_h, \rho^{(i)}_h) = A(\rho^{(i)}_h, \phi_h)\\
= & \sum_{j=0}^i {i\choose j} A^{(i-j)}( \rho^{(j)}_h, \phi_h) - \sum_{j=0}^{i-1} {i\choose j} A^{(i-j)}( \rho^{(j)}_h, \phi_h) \\
= & \sum_{j=0}^i {i\choose j} \left(A_h^{(i-j)}(\rho^{(j)}_h, \phi_h) +E\left(A^{(i-j)}( \rho^{(j)}_h, \phi_h)\right)\right) - \sum_{j=0}^{i-1} {i\choose j} \left(\rho^{(j)}_h, (L^*)^{(i-j)}\phi_h \right).
\end{align*}

Note that $\forall \chi \in V^h_0$, with \eqref{dt-A-rho} and \eqref{epsilon-m-estimate},
\begin{equation}\label{rho-0-square-estimate}
\begin{aligned}
& \sum_{j=0}^i {i\choose j} A_h^{(i-j)}(\rho^{(j)}_h, \phi_h) \\
= & \sum_{j=0}^i {i\choose j} A_h^{(i-j)}(\rho^{(j)}_h, \phi_h-\chi) + \sum_{j=0}^i {i\choose j} A_h^{(i-j)}(\rho^{(j)}_h, \chi)\\
= & \sum_{j=0}^i {i\choose j} A_h^{(i-j)}(\rho^{(j)}_h, \phi_h-\chi) + \epsilon^{(i)}( \chi) \\
\leq & C \sum_{j=0}^i \|\rho^{(j)}_h( t)\|_1\| \phi_h-\chi\|_1 +\mathcal O(h^{k+2})\left(\sum_{j=0}^i\|u^{(j)}( t)\|_{k+3}+\|(Lu)^{(i)}( t)\|_{k+2}\right)\|\chi\|_2.
\end{aligned}
\end{equation}

Let $\chi = \Pi_1 \phi_h$ where $\Pi_1$ is the $L^2$ projection to  functions in the  continuous piecewise $Q^1$ polynomial space, see \cite{li2020superconvergence}. Then we have $\| \phi_h-\chi\|_1 \leq Ch \|\phi_h\|_2$ and $\|\chi\|_2 \leq C\|\phi_h\|_2$.  Inserting  \eqref{rho-1-estimate} and \eqref{dual-elliptic-regularity} into \eqref{rho-0-square-estimate}, we have 
\begin{equation}\label{1st-term-in-rho-0-estimate}
\sum_{j=0}^i {i\choose j} A_h^{(i-j)}(\rho^{(j)}_h, \phi_h) = \mathcal O(h^{k+2})\left(\sum_{j=0}^i(\|u^{(j)} t)\|_{k+3}+\|(Lu)^{(i)}( t)\|_{k+2}\right)\|\phi_h\|_2.
\end{equation}
Thus with \eqref{1st-term-in-rho-0-estimate}, Lemma \ref{bilinear-quadrature-error}, and inverse inequality we have
\begin{equation}\label{rho-i-0-norm-estimation}
\begin{aligned}
 &  \|\rho^{(i)}_h(t)\|_0^2  \\
\leq & \mathcal O(h^{k+2})\left(\sum_{j=0}^i\|u^{(j)}( t)\|_{k+3}+\|(Lu)^{(i)}( t)\|_{k+2}\right)\|\phi_h\|_2 \\
&+ \mathcal O(h^{k+2})\sum_{j=0}^i\|\rho_h^{(j)}( t)\|_{k+2}\|\phi_h\|_2
  + C\sum_{j=0}^{i-1}\|\rho_h^{(j)}( t)\|_0\|\phi_h\|_2\\
= & \left[\mathcal O(h^{k+2})\left(\sum_{j=0}^i\|u^{(j)}\|_{k+3}+\|(Lu)^{(i)}\|_{k+2} \right) + C\sum_{j=0}^{i-1}\|\rho_h^{(j)}(t)\|_0\right]\|\phi_h\|_2\\
\leq & \left(\mathcal O(h^{k+2})\left(\sum_{j=0}^i\|u^{(j)}\|_{k+3}+\|(Lu)^{(i)}\|_{k+2} \right) + C\sum_{j=0}^{i-1}\|\rho_h^{(j)}(t)\|_0\right)\|\rho^{(i)}_h(t)\|_0,
\end{aligned}
\end{equation}
where \eqref{dual-elliptic-regularity} is applied in the last inequality.

 With similar induction arguments as above,  \eqref{rho-i-0-norm-estimation} implies 
 \begin{equation}\label{intermediate-rho-0-estimate}
 \|\rho^{(i)}_h( t)\|_0 \leq \mathcal O(h^{k+2})\sum_{j=0}^i(\|u^{(j)}( t)\|_{k+3}+\|(Lu)^{(j)}( t)\|_{k+2}).
 \end{equation}
 
Take the square for both sides of \eqref{intermediate-rho-0-estimate} then integrate from $0$ to $T$ and take the square root for both sides, we can get \eqref{rho-m-l2-estimate}. 
Take the maximum of the right hand side then the left hand side of \eqref{intermediate-rho-0-estimate} for $t\in[0,T]$, we can get \eqref{rho-m-infty-estimate}. 
\end{proof}

\section{ Accuracy of the semi-discrete schemes}\label{error-estimate}

In this section, we will prove the $(k+2)$-th order of accuracy of $Q^k$ spectral element method, when the errors are measured only at nodes of degree of freedoms, which is a superconvergence result of function values. 

 Throughout this section the generic constant $C$ is independent of $h$. Although in principle it may depend on $t$ though the coefficients $a_{ij}(t)$, $b_{j}(t)$, $c(t)$, we also treat it as independent of time since its time dependent version can always be replaced by a time independent constant after taking maximum over the ime interval $[0,T]$.    
 In what follows we will state and prove the main theorems for wave, parabolic and the Schr\"{o}dinger equations. 

\subsection{The hyperbolic problem}\label{error-wave}
 The   main  result for the wave equation can be stated as the following theorem. 
\begin{theorem}\label{wave-u-uh}
If $a_{ij}$, $b_j$, $c\in C^2\left([0,T];W^{k+2,\infty}(\Omega)\right)$, $u\in C^{2}\left([0,T];H^{k+4}(\Omega)\right)$, then for the semi-discrete scheme \eqref{wave-scheme} for the problem \eqref{wave-equation-homo-bc},  we have
\begin{equation*}
    \resizebox{\textwidth}{!}
     {
$
 \begin{aligned}
\|u_h-u\|_{L^{2}([0,T];l^2(\Omega))} \leq & C h^{k+2}\left(\sum_{j=0}^2(\|u^{(j)}\|_{L^{2}([0,T];H^{k+3}(\Omega))}+\|(Lu)^{(j)}\|_{L^{2}([0,T];H^{k+2}(\Omega))})\right.\\
&\qquad\qquad+ \left. \sum_{j=0}^1(\|u^{(j)}( 0)\|_{k+3}+\|(Lu)^{(j)}( 0)\|_{k+2})\right),\\
\|u_h-u\|_{L^{\infty}([0,T];l^2(\Omega))} \leq & C h^{k+2} \sum_{j=0}^2(\|u^{(j)}\|_{L^{\infty}([0,T];H^{k+3}(\Omega))}+\|(Lu)^{(j)}\|_{L^{\infty}([0,T];H^{k+2}(\Omega))}), 
\end{aligned}$     }
\end{equation*}
where  $C$ is independent of $t$, $h$, $u$, and $f$.
\end{theorem}
\begin{proof}
Note  that  for  the  numerical solution $u_h$ we have
\begin{equation}
\langle  u_h^{(2)}, v_h \rangle_h + A_h(u_h, v_h) = \langle f, v_h \rangle_h,\quad  \forall v_h \in V^h_0.
\end{equation} 
The exact solution $u$ satisfies $u_{tt}=-Lu+f$ thus the elliptic projection \eqref{elliptic-proj} satisfies
\[A_h(R_h u,v_h) = \langle u^{(2)}-f,v_h \rangle_h,\quad \forall v_h\in V^h_0.\]
Subtracting the two equations above, we get $ \theta_h=u_h-R_hu$, which satisfies 
\begin{equation}\label{theta-eqn-1}
\langle \theta_h^{(2)}, v_h \rangle_h + A_h(\theta_h, v_h) = -\langle \rho_h^{(2)},  v_h\rangle_h + 
\langle u^{(2)}-u^{(2)}_p,v_h \rangle,\quad \forall v_h \in V^h_0.
\end{equation}
 Note  that  
\begin{align}\label{dt-bilinear-form}
\frac{d}{dt}A_h(\theta_h,\theta_h) = A^{(1)}_{h}(\theta_h, \theta_h)+2A_h(\theta_h,\theta_h^{(1)})- \langle  \mathbf b\cdot\nabla\theta_h ,\theta_h^{(1)}\rangle_h + \langle  \mathbf b\cdot\nabla\theta_h^{(1)} ,\theta_h\rangle_h.
\end{align}
Thus by Lemma \ref{lemma-quaderror-2norm} and \eqref{norm-equivalence}, we have 
\begin{equation}\label{extra-convection-estimate}
\begin{aligned}
\langle  \mathbf b\cdot\nabla\theta_h^{(1)} ,\theta_h\rangle_h
= & (\mathbf b \cdot \nabla \theta_h^{(1)}, \theta_h) + \mathcal O(h^2)|\mathbf b \theta_h|_2\|\nabla  \theta_h^{(1)}\|_0 \\
\leq & (\mathbf b \cdot \nabla \theta_h^{(1)}, \theta_h) + C\| \theta_h^{(1)}\|_0\|\theta_h\|_1\\
= & (\nabla \cdot(\mathbf b \theta_h),\theta_h^{(1)}) + C\| \theta_h^{(1)}\|_0\|\theta_h\|_1\\
\leq & C \| \theta_h^{(1)}\|_0\|\theta_h\|_1\leq  C \| \theta_h^{(1)}\|_{l^2}\|\theta_h\|_1,
\end{aligned}
\end{equation}
where  an  inverse inequality  was applied to the first inequality and integration by parts in  $\theta_h\in V_0^h$  yields   the last equation. 

 Next we estimate  $\|\theta_h^{(1)}(s)\|^2_{0} +\|\theta_h(s)\|^2_{1}$. 
Take $v_h = \theta_h^{(1)} $ in \eqref{theta-eqn-1}  and  integrate with respect to $t$ from $0$ to $s$. With \eqref{dt-bilinear-form}, we have 
\begin{equation}\label{theta-estimate}
\resizebox{\textwidth}{!}{$
\begin{aligned}
& \int_0^s\frac{d}{dt}\left(\frac{1}{2}\langle \theta_h^{(1)}, \theta_h^{(1)} \rangle_h + \frac12 A_h(\theta_h,\theta_h)\right)dt\\
 = & \frac12 \int_0^s A^{(1)}_h(\theta_h,\theta_h)-\langle  \mathbf b\cdot\nabla\theta_h ,\theta_h^{(1)}\rangle_h +\langle  \mathbf b\cdot\nabla\theta_h^{(1)} ,\theta_h\rangle_h - 2\langle \rho_h^{(2)}, \theta_h^{(1)}\rangle_h + 2\langle u^{(2)}-u^{(2)}_p, \theta_h^{(1)} \rangle_h dt.
\end{aligned}$}
\end{equation}
 With $\theta_h(0) = 0$ and
\eqref{extra-convection-estimate}, this  implies
\begin{equation}\label{theta-estimate-0} 
\begin{aligned}
 & \frac12(\|\theta_h^{(1)}(s)\|^2_{l^2}+A_h(\theta_h(s),\theta_h(s))) - \frac12\|\theta_h^{(1)}(0)\|^2_{l^2} \\
\leq & C\int_0^s  (\|\theta_h\|^2_{1} +  \| \theta_h^{(1)}\|_0\|\theta_h\|_1)dt +C\int_0^s \|\rho_h^{(2)}\|_{0} \|\theta_h^{(1)}\|_{0} dt\\
& + C\int_0^s  \| u^{(2)}-u^{(2)}_p\|_{l^2}\|\theta_h^{(1)}\|_{0}dt\\
\leq & C\int_0^s( \| \theta_h^{(1)}\|^2_0 + \|\theta_h\|^2_{1}) dt +C\int_0^s (\|\rho_h^{(2)}\|^2_{0} + \| u^{(2)}-u^{(2)}_p\|^2_{l^2})  dt,
\end{aligned}
\end{equation}
where Cauchy-Schwarz inequality  was  used in the last inequality.

Thus with \eqref{norm-equivalence}, \eqref{vh-ellipticity}, and \eqref{theta-estimate-0} we have
\begin{equation}\label{theta-estimate-1} 
\begin{aligned}
& \|\theta_h^{(1)}(s)\|^2_{0} +\|\theta_h(s)\|^2_{1} \leq C\|\theta_h^{(1)}(s)\|^2_{l^2}  + C A_h(\theta_h(s),\theta_h(s))\\
\leq & C\|\theta_h^{(1)}(0)\|^2_{l^2} + C\int_0^s( \| \theta_h^{(1)}\|^2_0 + \|\theta_h\|^2_{1}) dt +C\int_0^s (\|\rho_h^{(2)}\|^2_{0} + \| u^{(2)}-u^{(2)}_p\|^2_{l^2})  dt.
\end{aligned}
\end{equation}
With the  Gronwall  inequality \eqref{gronwall-inequality} we can eliminate the second term to find 
\begin{align*}
\|\theta_h^{(1)}(s)\|^2_{0} +\|\theta_h(s)\|^2_{1} \leq C\|\theta_h^{(1)}(0)\|^2_{l^2}  +C\int_0^s \|\rho_h^{(2)}\|^2_{0} + \| u^{(2)}-u^{(2)}_p\|^2_{l^2}  dt.
\end{align*}
With \eqref{rho-m-infty-estimate} and Theorem \ref{thm-superapproximation} we have
\begin{equation*}\label{u-up-estimate-in-energy}
\resizebox{\textwidth}{!}
     {$
\begin{aligned}
\|\theta_h^{(1)}(s)\|^2_{0} +\|\theta_h(s)\|^2_{1} \leq C\|\theta_h^{(1)}(0)\|^2_{l^2}  +\mathcal O(h^{2k+4})\int_0^s \sum_{j=0}^2(\|u^{(j)}\|_{k+3}+\|(Lu)^{(j)}\|_{k+2})^2 dt,
\end{aligned}$}
\end{equation*}
i.e.
\begin{equation}\label{theta-estimate-3} 
\begin{aligned}
 \|\theta_h^{(1)}(s)\|_{0} +\|\theta_h(s)\|_{1}
\leq C\|\theta_h^{(1)}(0)\|_{l^2} +\mathcal O(h^{k+2})\int_0^s \sum_{j=0}^2(\|u^{(j)}\|_{k+3}+\|(Lu)^{(j)}\|_{k+2}) dt.
\end{aligned}
\end{equation}
 To estimate $\|\theta_h^{(1)}(0)\|_{l^2}$ we use  Theorem \ref{thm-superapproximation},  \eqref{rho-m-infty-estimate}, and \eqref{norm-equivalence},
\begin{align*}
\|\theta_h^{(1)}(0)\|_{l^2} = & \| (u_1)_I- (R_hu)^{(1)}(0)\|_{l^2} \\
 = & \| (u_1)_I-(u_1)_p + (u_1)_p -(R_hu)^{(1)}(0)\|_{l^2} \\
 \leq & \| (u_1)_I-(u_1)_p \|_{l^2}  + \|(u_1)_p -(R_hu)^{(1)}(0)\|_{l^2} \\
 = & \| u_1-(u_1)_p \|_{l^2}  + \|(u_1)_p -R_h(u^{(1)}(0))\|_{l^2} \\
 = & \| u_1-(u_1)_p \|_{l^2}  + \|(u_1)_p -R_h(u_1)\|_{l^2} \\
 = & \mathcal O(h^{k+2})(\|u_1\|_{k+3}+\|Lu_1\|_{k+2}).
\end{align*}
Then we have
\begin{equation}
\begin{aligned}
& \|\theta_h^{(1)}\|_{0} +\|\theta_h\|_{1} \\
\leq & \mathcal O(h^{k+2})  \left( \|u_1\|_{k+3}+\|Lu_1\|_{k+2} + \int_0^s \sum_{j=0}^2(\|u^{(j)}\|_{k+3}+\|(Lu)^{(j)}\|_{k+2}) dt \right).
\end{aligned}
\end{equation}
Now with \eqref{rho-m-l2-estimate},  \eqref{rho-m-infty-estimate}, and Theorem \ref{thm-superapproximation}, the proof is concluded.
\end{proof}

\subsection{The parabolic problem}\label{error-heat}
 We now present  the main  result for the parabolic problem. 
 \begin{theorem}\label{heat-u-uh}
If $a_{ij}$, $b_j$, $c\in C^1([0,T];W^{k+1,\infty}(\Omega))$, $u\in C^{1}([0,T];H^{k+4}(\Omega))$, then for the semi-discrete scheme \eqref{heat-scheme} for problem \eqref{heat-equation-homo-bc}, we have
\begin{equation*}
\resizebox{\textwidth}{!}
     {$
\begin{aligned}
\|u_h-u\|_{L^{2}([0,T];l^2(\Omega))} \leq &  C h^{k+2} \sum_{j=0}^1(\|u^{(j)}\|_{L^{2}([0,T];H^{k+3}(\Omega))}+\|(Lu)^{(j)}\|_{L^{2}([0,T];H^{k+2}(\Omega))}),\\
\|u_h-u\|_{L^{\infty}([0,T];l^2(\Omega))} \leq & C h^{k+2} \sum_{j=0}^1(\|u^{(j)}\|_{L^{\infty}([0,T];H^{k+3}(\Omega))}+\|(Lu)^{(j)}\|_{L^{\infty}([0,T];H^{k+2}(\Omega))}),
\end{aligned}$}
\end{equation*}
where  $C$ is independent of $t$, $h$, $u$, and $f$.
\end{theorem}
\begin{proof}
By our semi-discrete numerical scheme \eqref{heat-scheme} and the definition of  the  elliptic projection \eqref{elliptic-proj}, we have 
\begin{equation}\label{heat-theta-eqn-1}
\langle \theta_h^{(1)}, v_h \rangle_h + A_h(\theta_h, v_h) = -\langle \rho_h^{(1)},  v_h\rangle_h + 
\langle u^{(1)}-u_p^{(1)},v_h \rangle,\quad \forall v_h \in V^h_0.
\end{equation}
 
Take $v_h = \theta_h^{(1)} $ in \eqref{heat-theta-eqn-1}  and  integrate with respect to $t$ from $0$ to $s$, 
\begin{equation}\label{heat-theta-estimate} 
\resizebox{\textwidth}{!}
     {$
\begin{aligned}
& \int_0^s\langle \theta_h^{(1)}, \theta_h^{(1)} \rangle_h + \frac12 \frac{d}{dt}A_h(\theta_h,\theta_h)dt\\
 = & \frac12 \int_0^s A^{(1)}_h(\theta_h,\theta_h)-\langle  \mathbf b\cdot\nabla\theta_h ,\theta_h^{(1)}\rangle_h +\langle  \mathbf b\cdot\nabla\theta_h^{(1)} ,\theta_h\rangle_h - 2\langle \rho_h^{(1)}, \theta_h^{(1)}\rangle_h + 2\langle u^{(1)}-u^{(1)}_p, \theta_h^{(1)} \rangle_h dt.
\end{aligned}$}
\end{equation}
Note that $\theta_h(0) = 0$, then with \eqref{norm-equivalence}, \eqref{extra-convection-estimate}, and \eqref{heat-theta-estimate}  we have
\begin{equation*}
\begin{aligned}
& \int_0^s\langle \theta_h^{(1)}, \theta_h^{(1)} \rangle_hdt +\|\theta_h(s)\|^2_{1} \leq \int_0^s\langle \theta_h^{(1)}, \theta_h^{(1)} \rangle_hdt + C A_h(\theta_h(s),\theta_h(s))\\
\leq & C\int_0^s  \|\theta_h\|^2_{1}dt + C\int_0^s \| \theta_h^{(1)}\|_{l^2}\|\theta_h\|_1dt +C\int_0^s \|\rho_h^{(1)}\|_{l^2} \|\theta_h^{(1)}\|_{l^2} dt \\
 & + C\int_0^s  \| u^{(1)}-u^{(1)}_p\|_{l^2}\|\theta_h^{(1)}\|_{l^2}dt\\
\leq & C\int_0^s  \|\theta_h\|^2_{1}dt + \int_0^s \epsilon\langle \theta_h^{(1)}, \theta_h^{(1)} \rangle_h +  \frac{C}{4\epsilon}\|\theta_h\|^2_1dt +\int_0^s \epsilon\langle \theta_h^{(1)}, \theta_h^{(1)} \rangle_h+\frac{C}{4\epsilon}\|\rho_h^{(1)}\|^2_{0}  dt \\
 & + \int_0^s  \epsilon\langle \theta_h^{(1)}, \theta_h^{(1)} \rangle_h+\frac{C}{4\epsilon}\| u^{(1)}-u^{(1)}_p\|^2_{l^2}dt,
\end{aligned}
\end{equation*}
where Cauchy-Schwartz inequality  was  applied in the last inequality.
Thus we have
\begin{equation*}
\begin{aligned}
 (1-3\epsilon)\int_0^s\langle \theta_h^{(1)}, \theta_h^{(1)} \rangle_hdt +\|\theta_h(s)\|^2_{1} \leq & C(1+\frac{1}{4\epsilon}) \int_0^s  \|\theta_h\|^2_{1}dt +  \frac{C}{4\epsilon}\int_0^s \|\rho_h^{(1)}\|^2_{0}dt \\
&  + \frac{C}{4\epsilon}\int_0^s \| u^{(1)}-u^{(1)}_p\|^2_{l^2}dt.
\end{aligned}
\end{equation*}
 Now take  $\epsilon$ small enough to make $1-3\epsilon\geq \frac12$ then 
\begin{equation}\label{heat-theta-estimate-1} 
\begin{aligned}
 \frac12 \int_0^s\langle \theta_h^{(1)}(s), \theta_h^{(1)} \rangle_h(s)dt +\|\theta_h(s)\|^2_{1} \leq  C\int_0^s \|\rho_h^{(1)}\|^2_{0}dt + C\int_0^s \| u^{(1)}-u^{(1)}_p\|^2_{l^2}dt\\
+ C\int_0^s \left( \|\theta_h(t)\|^2_{1}+ \frac12 \int_0^t\langle \theta_h^{(1)}(\eta), \theta_h^{(1)}(\eta) \rangle_h d\eta\right)dt. 
\end{aligned}
\end{equation}
 Next, apply  Gronwall's inequality to eliminate the last term of the right hand side of \eqref{heat-theta-estimate-1}  to find 
\begin{equation*}
\frac12\int_0^s\langle \theta_h^{(1)}, \theta_h^{(1)} \rangle_hdt + \|\theta_h\|^2_{1} \leq C \int_0^s \|\rho_h^{(1)}\|^2_{0}dt + C\int_0^s \| u^{(1)}-u^{(1)}_p\|^2_{l^2}dt.
\end{equation*}
 Using \eqref{rho-m-l2-estimate},  \eqref{rho-m-infty-estimate}, and Theorem \ref{thm-superapproximation} we have 
\begin{align*}
\frac12\int_0^s\langle \theta_h^{(1)}, \theta_h^{(1)} \rangle_hdt + \|\theta_h\|^2_{1} \leq \mathcal O(h^{k+2}) \int_0^s \sum_{j=0}^1(\|u^{(j)}\|_{k+3}+\|(Lu)^{(j)}\|_{k+2}) dt,
\end{align*}
 concluding the proof.
\end{proof}

\subsection{The linear Schr\"{o}dinger equation}\label{error-schrodinger}
Consider  the  problem 
\begin{equation}\label{schrodinger-eqn}
 \left\{\begin{array}{ll}
i u_{t}=-\Delta u+V u+f, & \textrm{ in } \Omega \times[0, T], \\
u(\mathbf x, t)=0, & \textrm{ on } \partial \Omega \times[0, T], \\
u(\mathbf x, 0)=u_{0}(\mathbf x), &  \textrm{ in } \Omega,
\end{array}\right.
\end{equation}
where $\Omega \in R^{2}$ is a rectangular domain, the functions $u_{0}(\mathbf x), f(\mathbf x, t)$, and the solution $u(\mathbf x, t)$ are complex-valued  while  the potential function $V(\mathbf x,t)$ is real-valued,  non-negative, and  bounded for all $(\mathbf x,t) \in \Omega \times[0, T]$.  

 In this subsection we work with complex-valued functions and  the definition of inner product and the induced norms are modified accordingly. For instance,  for complex-valued $v$, $w$ $\in L^2(\Omega)$, the inner product is defined as
$$
(v, w) :=  \int_{\Omega} v\bar{w} d\mathbf x.
$$
We assume all the functions of the function spaces defined previously are complex-valued for this subsection, such as $H^k(\Omega)$, $H_0^k(\Omega)$, $V^h_0$, etc.

The   variational form of \eqref{schrodinger-eqn} is: for $t \in[0, T],$ find $u(t) \in H_{0}^{1}(\Omega)$ satisfying:
\begin{equation}\label{schrodinger-equation-homo-bc}
\left\{\begin{array}{ll}
i \left(u_{t}, v\right) - (\nabla u, \nabla v) - (Vu,  v)  = (f, v), & \forall v \in H_{0}^{1}(\Omega), \\
u(0)=u_{0},  & \forall v \in H_{0}^{1}(\Omega).
\end{array}\right.
\end{equation}
The semi-discrete numerical scheme discretizing  \eqref{schrodinger-equation-homo-bc} is to find $u_h \in V^h_0$ satisfying
\begin{equation}\label{schrodinger-scheme}
\left\{\begin{array}{ll}
i\langle (u_h)_{t}, v_h \rangle_h - \langle \nabla u_h, \nabla v_h\rangle_h - \langle Vu_h,  v_h\rangle_h  = \langle f, v_h\rangle_h, & \forall v_h \in V^h_0, \\
u_h(0)=(u_{0})_I,
\end{array}\right.
\end{equation}
and the elliptic projection $ R_hu\in V_0^h$ is defined as 
\begin{equation}
\langle \nabla R_hu, \nabla v_h \rangle_h + \langle V R_hu,  v_h \rangle_h = \langle -\Delta u + Vu,  v_h \rangle_h, \quad \forall v_h \in V^h_0.
\end{equation}
 
As in Section \ref{elliptic-proj-error-estimate},  we split the error into two parts  
$$e=\theta_h + \rho_h,$$ 
where $\theta_h = u_h-R_h u\in V^h_0$ and $\rho_h = R_hu - u_p \in V^h_0$.  The estimates for $\rho_h^{(m)}$,  $m\geq 0$  from  Lemma \ref{rho-m-estimate}  are still valid.  

\begin{theorem}\label{schrodinger-u-uh}
If $u\in C^{1}([0,T];H^{k+4}(\Omega))$, then for the semi-discrete scheme \eqref{schrodinger-scheme} for problem \eqref{schrodinger-eqn}, we have
\begin{equation*}
\resizebox{\textwidth}{!}
     {$
\begin{aligned}
\|u_h-u\|_{L^{2}([0,T];l^2(\Omega))} \leq & C h^{k+2} \sum_{j=0}^1(\|u^{(j)}\|_{L^{2}([0,T];H^{k+3}(\Omega))}+\|(Lu)^{(j)}\|_{L^{2}([0,T];H^{k+2}(\Omega))}),\\
\|u_h-u\|_{L^{\infty}([0,T];l^2(\Omega))} \leq & C h^{k+2} \sum_{j=0}^1(\|u^{(j)}\|_{L^{\infty}([0,T];H^{k+3}(\Omega))}+\|(Lu)^{(j)}\|_{L^{\infty}([0,T];H^{k+2}(\Omega))}),
\end{aligned}$}
\end{equation*}
where   $C$ is independent of $t$, $h$, $u$, and $f$.
\end{theorem}
\begin{proof}

 As in the parabolic case we start by estimating $\theta_h$. 
\begin{equation}\label{schrodinger-theta-eqn-1}
\langle \theta_h^{(1)}, v_h \rangle_h + i\langle \nabla \theta_h, \nabla v_h\rangle_h +i\langle V\theta_h,  v_h\rangle_h = -\langle \rho_h^{(1)},  v_h\rangle_h + 
\langle u^{(1)}-u_p^{(1)},v_h \rangle_h,\quad \forall v_h \in V^h_0.
\end{equation}
Taking $v_h =\theta_h$ in \eqref{schrodinger-theta-eqn-1}  and taking real part,
\begin{align*}
\frac{d}{dt}\|\theta_h\|^2_{l^2(\Omega)}  = \frac{d}{dt}\langle \theta_h,\theta_h \rangle_h = & 2 \textit{Re}\left(-\langle \rho_h^{(1)},  \theta_h\rangle_h + 
\langle u^{(1)}-u_p^{(1)}, \theta_h \rangle_h\right) \\
\leq & 2\left(\| \rho_h^{(1)}\|_{l^2(\Omega)}+\|u^{(1)}-u_p^{(1)}\|_{l^2(\Omega)}\right)\|\theta_h\|_{l^2(\Omega)}.
\end{align*}
 Since $\frac{d}{dt}\|\theta_h\|^2_{l^2(\Omega)} =2 \|\theta_h\|_{l^2(\Omega)} \frac{d}{dt}\|\theta_h\|_{l^2(\Omega)}$, it impilies \begin{align*}
\frac{d}{dt}\|\theta_h\|_{l^2(\Omega)} \leq  \| \rho_h^{(1)}\|_{l^2(\Omega)}+\|u^{(1)}-u_p^{(1)}\|_{l^2(\Omega)}.
\end{align*}
 Upon integrating this inequality with respect to $t$ from $0$ to $s$ we have 
\begin{align*}
\|\theta_h(s)\|_{l^2(\Omega)} \leq  \|\theta_h(0)\|_{l^2(\Omega)} + \int_0^s (\| \rho_h^{(1)}\|_{l^2(\Omega)}+\|u^{(1)}-u_p^{(1)}\|_{l^2(\Omega)} )dt.
\end{align*}
 Now, using Theorem \ref{thm-superapproximation},  \eqref{rho-m-infty-estimate},  and \eqref{norm-equivalence}  we have
\begin{align*}
\|\theta_h(0)\|_{l^2} = & \| (u_0)_I- (R_hu)(0)\|_{l^2} \\
 = & \| (u_0)_I-(u_0)_p + (u_0)_p -(R_hu)(0)\|_{l^2} \\
 \leq & \| (u_0)_I-(u_0)_p \|_{l^2}  + \|(u_0)_p -(R_hu)(0)\|_{l^2} \\
 = & \| u_0-(u_0)_p \|_{l^2}  + \|(u_0)_p -R_hu_0\|_{l^2} \\
 = & O(h^{k+2})(\|u_0\|_{k+3}+\|Lu_0\|_{k+2}).
\end{align*}

With this result in concert with \eqref{rho-m-l2-estimate}, \eqref{rho-m-infty-estimate}, and Theorem \ref{thm-superapproximation} we note 
\begin{equation*}
    \resizebox{\textwidth}{!}
     {$
\begin{aligned}
\|\theta_h(s)\|_{l^2(\Omega)} \leq \mathcal O(h^{k+2})\left(\|u_0\|_{k+3}+\|Lu_0\|_{k+2} +  \int_0^s \sum_{j=0}^1(\|u^{(j)}\|_{k+3}+\|(Lu)^{(j)}\|_{k+2}) dt \right).
\end{aligned}$}
\end{equation*}
 Together with \eqref{rho-m-l2-estimate}, \eqref{rho-m-infty-estimate}, and Theorem \ref{thm-superapproximation}, proof is concluded. 
\end{proof}

 \subsection{Neumann boundary conditions and $\ell^\infty$-norm estimate}
 \label{sec-Neumann}
 
 For homogeneous Neumann type boundary conditions, due to Lemma \ref{bilinear-u-up},  in general we can only prove $(k+\frac32)$-th order accuracy for the hyperbolic equation, parabolic equation, and linear Schr\"{o}dinger equation. As explained in Remark \ref{rmk-loss}, the half order loss
happens for homogeneous Neumann boundary condition only when the second order operator coefficient $\mathbf a$ is not diagonal, e.g., when the PDE contains second order mixed derivatives. If  $\mathbf a$ is diagonal, then all results of $(k+2)$-th order in $\ell^2$ norm in this Section can be easily extended to the homogeneous Neumann boundary conditions.  See Section 2.8 in \cite{Li2021} for a detailed discussion of nonhomogeneous Neumann boundary conditions. 
 
 For Lagrangian $Q^k$ finite element method without any quadrature solving the elliptic equation with Dirichlet boundary conditions,  the best superconvergence order in max norm of function values at Gauss-Lobatto that one can prove is   $\mathcal O(|\log h| h^{k+2})$ in two dimensions, see \cite{li2020superconvergence} and references therein. 
Thus we do not expect better results can be proven in 
the $Q^k$ spectral element method in $\ell^\infty$ norm over all nodes of degree of freedoms.  

\section{The implementation for nonhomogeneous Dirichlet boundary conditions}\label{implementation-nonhom}
Consider  the  hyperbolic problem on $\Omega=(0,1)^2$ with compatible nonhomogeneous Dirichlet boundary condition and initial value
\begin{equation}\label{wave-equation-nonhomo-bc}
\begin{aligned}
u_{tt}
 = & -Lu + f(\mathbf x,t) & \text{ in } \Omega \times (0,T], \\
u(\mathbf x,t)  = & g &\text{ on } \partial\Omega \times [0,T], \\
u(\mathbf x,0) = & u_0(\mathbf x) ,\quad u_t(\mathbf x,0) = u_1(\mathbf x) &\text{ on } \Omega\times\{t=0\}.
\end{aligned}
\end{equation}
As in \cite{gockenbach2006understanding, li2020superconvergence}, by abusing notation, we define
\[g(x,y,t)=\begin{cases}
   0,& \mbox{if}\quad (x,y)\in (0,1)\times(0,1),\\
   g(x,y,t),& \mbox{if}\quad (x,y)\in \partial\Omega,\\
  \end{cases}
\] 
and
define $g_I\in V^h$  as the $Q^k$ Lagrange interpolation at $(k+1)\times (k+1)$ Gauss-Lobatto points for each cell on $  \Omega$ of $g(x,y,t)$.
Namely, $g_I\in V^h$ is the piecewise $Q^k$ interpolant   of $g$ along $\partial\Omega$ at the boundary grid points and $g_I=0$ at the interior grid points. 
Then the semi-discrete scheme for problem \eqref{wave-equation-nonhomo-bc} is as follows: for $t\in[0,T]$, find $\tilde u_h\in V^h_0$ such that
 \begin{equation}\label{wave-scheme-nonhomo-bc}
\begin{aligned}
\langle \tilde u_h^{(2)}, v_h \rangle_h + A_h(\tilde u_h, v_h) = & \langle f, v_h \rangle_h- A_h(g_I,v_h),\quad  \forall v_h \in V^h_0,\\
\tilde u_h(0) = R_hu_0, \quad \tilde u_h^{(1)}(0) = & (u_1)_I.
\end{aligned}
\end{equation}
Then 
\begin{equation}
u_h := \tilde u_h + g_I,
\end{equation}
is the desired numerical solution. Notice that $u_h$ and $\tilde u_h$ are the same at all interior grid points. 

For the initial value of numerical solution, instead of using discrete elliptic projection, we can also use $\tilde u_h(0)=u(x,y,0)_I$ in \eqref{wave-scheme-nonhomo-bc} where  $u(x,y,0)_I$ is the piecewise Lagrangian $Q^k$ interpolation of $u(x,y,0)$. In all numerical tests in Section \ref{numerical-test}, $(k+2)$-th order accuracy is still observed for the initial condition $\tilde u_h(0)=u(x,y,0)_I$.

 The treatment for nonhomogeneous Dirichlet boundary condition above can be extended naturally to  the parabolic equation and linear Schr\"{o}dinger equation,

\begin{remark}
For the $(k+2)$-th order accuracy of the scheme \eqref{wave-scheme-nonhomo-bc}, it can be shown analogously as in \cite{li2020superconvergence}, and in Section \ref{elliptic-proj-error-estimate} and Section \ref{error-estimate} by defining discrete elliptic projection as
\begin{equation}\label{elliptic-proj-nonhomo}
R_h u := \tilde R_h u + g_I,
\end{equation}
where $\tilde R_hu \in V_0^h$ satisfying
\begin{equation*}
A_h(\tilde R_h u,v_h) = \langle -Lu,v_h \rangle_h- A_h(g_I,v_h),\quad \forall v_h\in V^h_0, \quad 0\leq t \leq T.
\end{equation*} 
\end{remark}

\section{Numerical examples}\label{numerical-test}
In this section we present numerical examples for the wave equation, a parabolic equation and the Schr\"{o}dinger equation. 
\subsection{Numerical examples for the wave equation}
\subsubsection{Timestepping} \label{sec:timestepping}
The so called modified equation technique, \cite{Dablain1986,37342,henshaw:1730,joly2010optimized}, is an attractive option for timestepping the scalar wave equation. 
After semidiscretization the method \eqref{wave-scheme} can be written as 
\[
\frac{d^2 {\bf u}_h}{dt^2} = Q {\bf u}_h,
\]
where ${\bf u}_h$ is a vector containing all the degrees of freedom and $Q$ is a matrix.  To evolve in time we expand the approximate solution around $t+\Delta t$ and $t-\Delta t$   
\[
{\bf u}_h(t+\Delta t) + {\bf u}_h(t-\Delta t) = 2 {\bf u}_h(t) + \Delta t^2 \frac{d^2 {\bf u}_h(t)}{dt^2} 
+ \frac{\Delta t^4}{12} \frac{d^4 {\bf u}_h(t)}{dt^4} 
+ \frac{\Delta t^6}{360} \frac{d^6 {\bf u}_h(t)}{dt^6} + \mathcal{O}(\Delta t^8).  
\]
Replacing the even time derivatives with applications of the matrix $Q$ we obtain, for example, a 6th order accurate explicit temporal approximation   
\[
{\bf u}_h(t+\Delta t) + {\bf u}_h(t-\Delta t) = 2 {\bf u}_h(t) + \Delta t^2 Q {\bf u}_h(t) 
+ \frac{\Delta t^4}{12} Q^2 {\bf u}_h(t) + \frac{\Delta t^6}{360} Q^3 {\bf u}_h(t).  
\]

 Note that the matrix $Q$ does not need to be explicitly known, and an implicit definition through a ``matrix-vector multiplication'' subroutine will suffice. In that case the three last terms on the right hand side of the above equation would be computed by repeated application of $Q$. For example to compute ${\bf u}_h(t+\Delta t)$ one would assign ${\bf v}_h = 2{\bf u}_h(t)-{\bf u}_h(t-\Delta t)$, ${\bf u}_h(t-\Delta t) = {\bf u}_h(t)$, followed by three applications of $Q$ and updates of ${\bf v}_h$: (1) ${\bf w}_h = Q {\bf u}_h(t)$, ${\bf v}_h \leftarrow {\bf v}_h +  \Delta t^2 {\bf w}_h$, ${\bf u}_h(t) = {\bf w}_h$, (2) ${\bf w}_h = Q {\bf u}_h(t)$, ${\bf v}_h \leftarrow {\bf v}_h +  \Delta t^4/12 {\bf w}_h$, ${\bf u}_h(t) = {\bf w}_h$, (3) ${\bf w}_h = Q {\bf u}_h(t)$, ${\bf v}_h \leftarrow {\bf v}_h +  \Delta t^6/360 {\bf w}_h$. The time update is then finalized by the assignment ${\bf u}_h(t) = {\bf v}_h$, which can conveniently be implemented as a for loop.  

\subsubsection{Standing mode with Dirichlet conditions} \label{sec:Dsquare_mode}

In this experiment we solve the the wave equation  $u_{tt}=u_{xx}+u_{yy}$  with homogenous Dirichlet boundary conditions in the square domain $(x,y) \in [-\pi,\pi]^2$. We take the initial data  to be 
\[
u(x,y,0) = \sin (x) \sin(y), \ \ u_t(x,y,0) = 0,
\] 
which results in the exact standing mode solution 
\[
u(x,y,0) = \sin (x) \sin(y) \cos(\sqrt{2}t).
\]

We consider the two cases $k = 2$ and $k=4$ and discretize on three different sequences of grids. The first sequence contains only plain Cartesian of increasing refinement. The second sequence consists of the same grids as in the Cartesian sequence but with all the interior nodes perturbed by a two dimensional uniform random variable with each component drawn from $[-h/4,h/4]$. The nodes of the third sequence are 
\[
(x,y) =  (\xi + 0.1 \sin(\xi) \sin(\eta),\eta + 0.1 \sin(\eta) \sin(\xi)), \ \ \ \ (\xi,\eta) = [-\pi,\pi]^2,
\] 
and this is refined in the same ways as the Cartesian sequence. 
Typical examples of the grids are displayed in Figure \ref{fig:square_mode_grids}.
 Even though the equation contains no coefficients, 
variable coefficients are still involved for the second and the third sequences of grids. The variable coefficients are induced by the geometric transformations of the elements in the mesh to a reference rectangle element. However, on a randomly perturbed grid, the variable coefficients are not smooth across cell interfaces. The variable coefficients are smooth in a smoothly perturbed grid.  

\begin{figure}[htb]
\begin{center}
\includegraphics[width=0.45\textwidth]{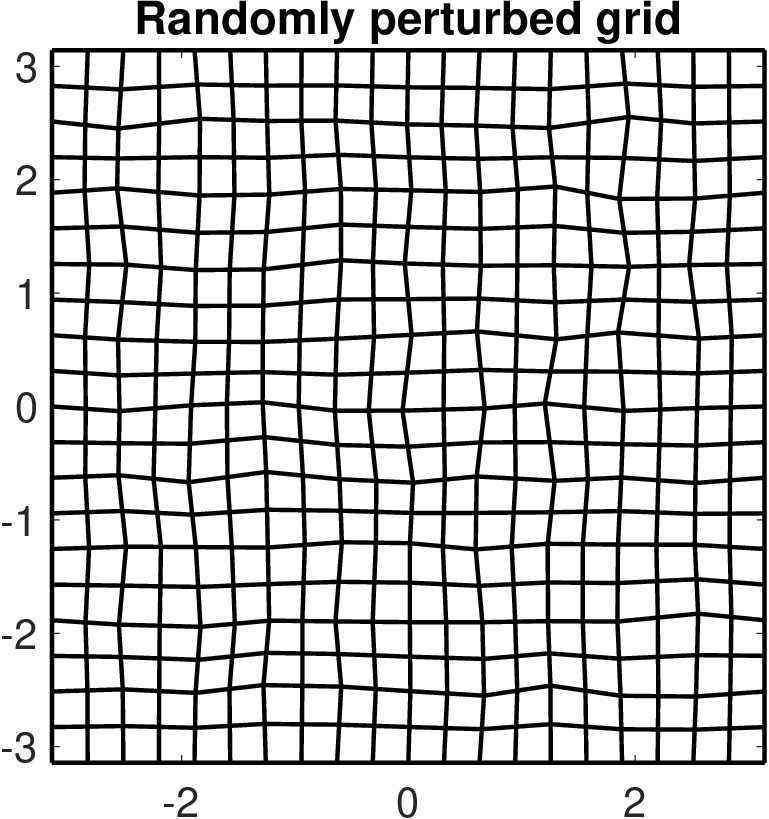}
\includegraphics[width=0.45\textwidth]{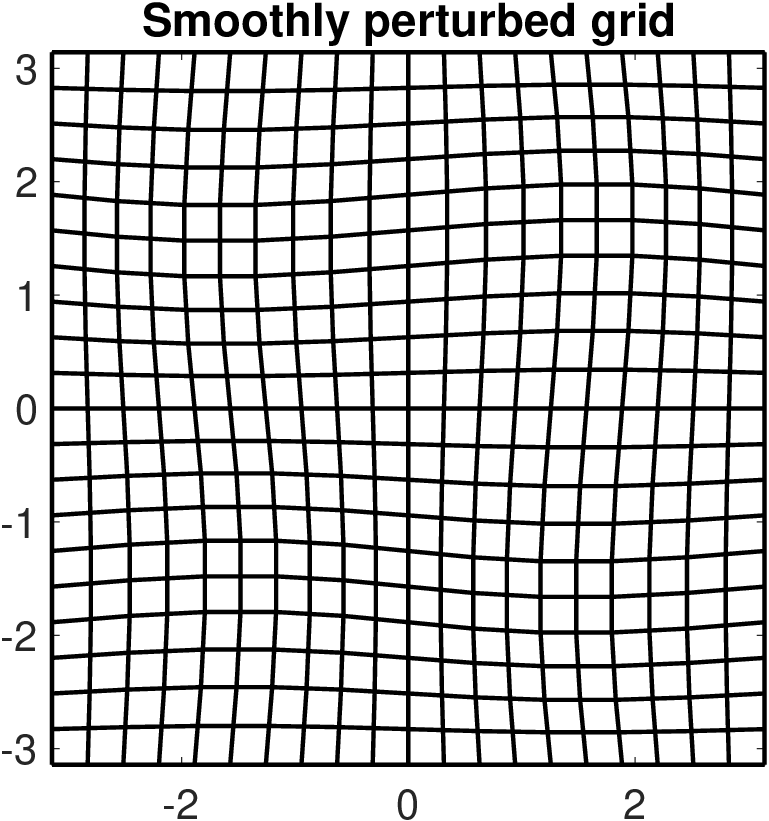}
\caption{Two typical grids used in the numerical examples in Section \ref{sec:Dsquare_mode} and \ref{sec:square_mode}. \label{fig:square_mode_grids}}
\end{center}
\end{figure}

\begin{figure}[htb]
\begin{center}
\includegraphics[width=0.45\textwidth]{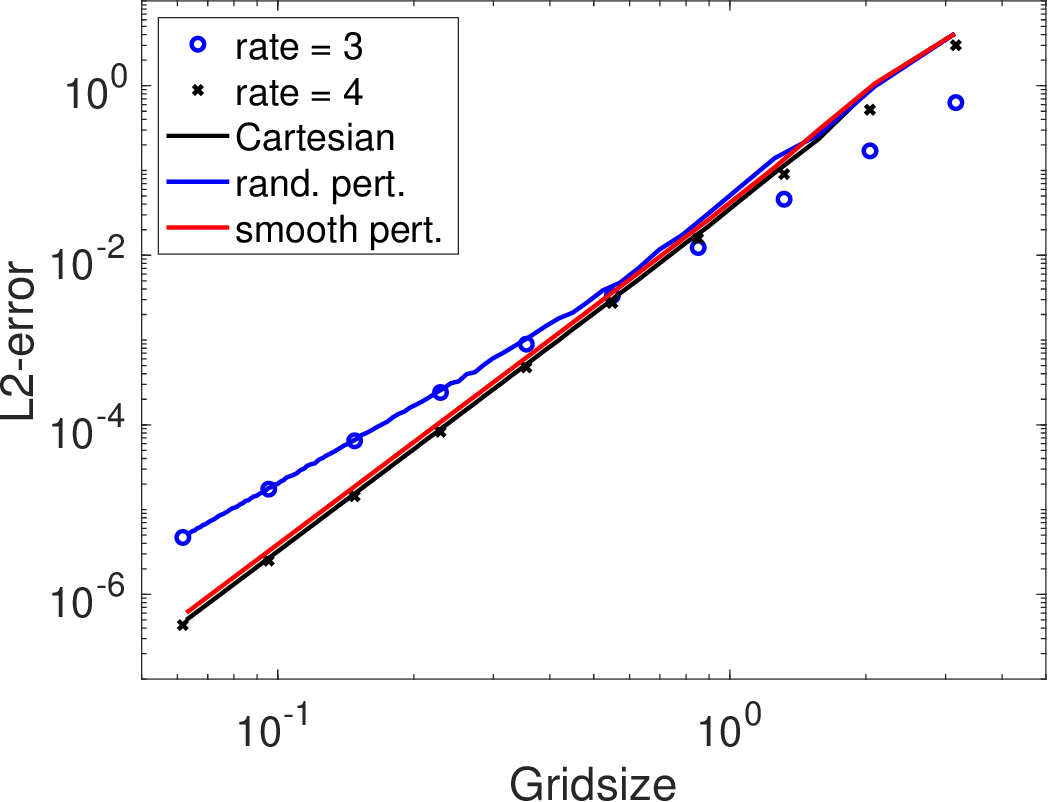}
\includegraphics[width=0.45\textwidth]{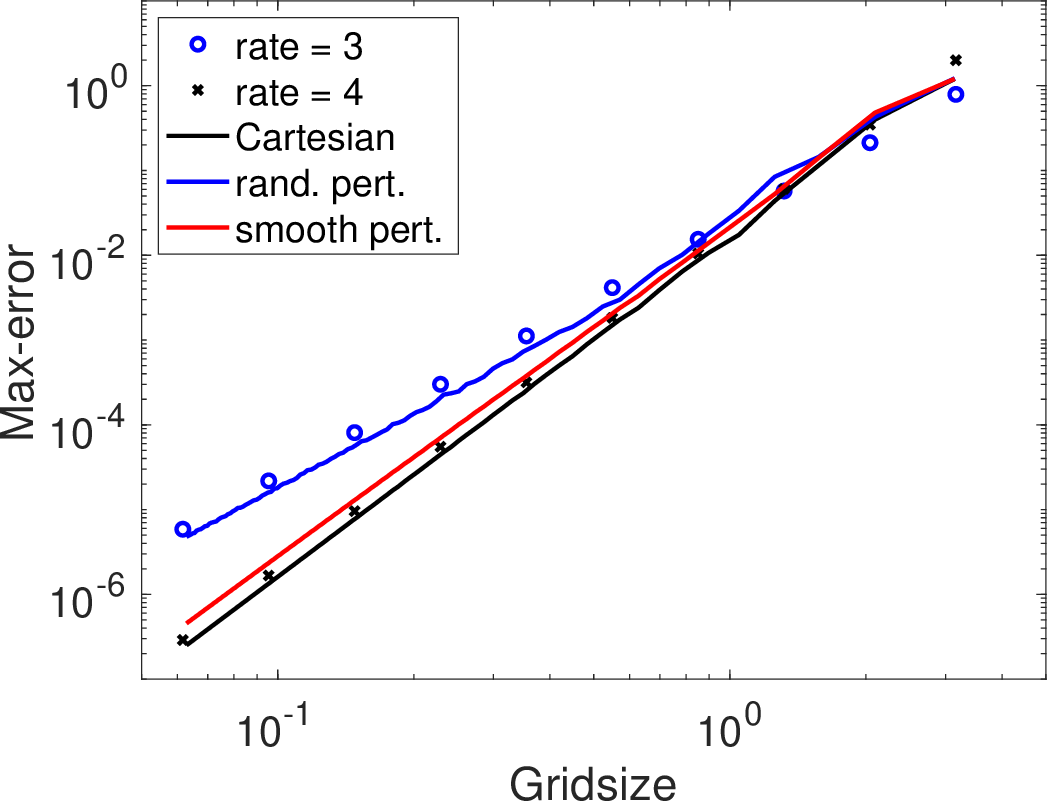}
\includegraphics[width=0.45\textwidth]{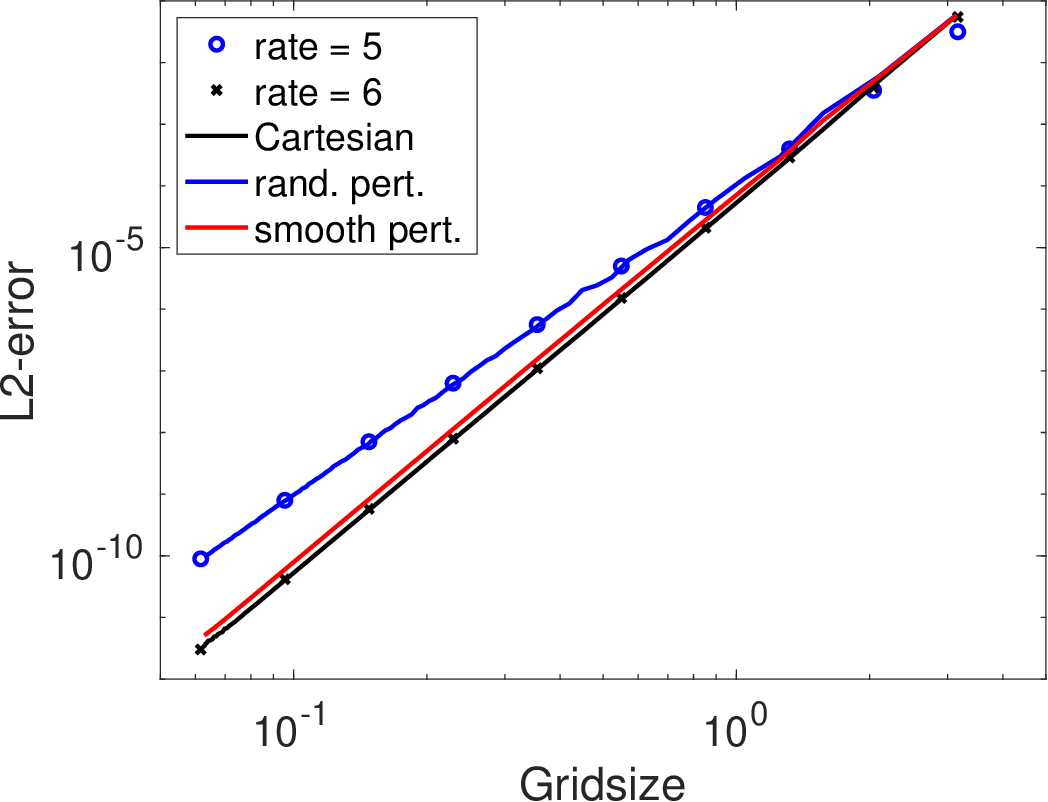}
\includegraphics[width=0.45\textwidth]{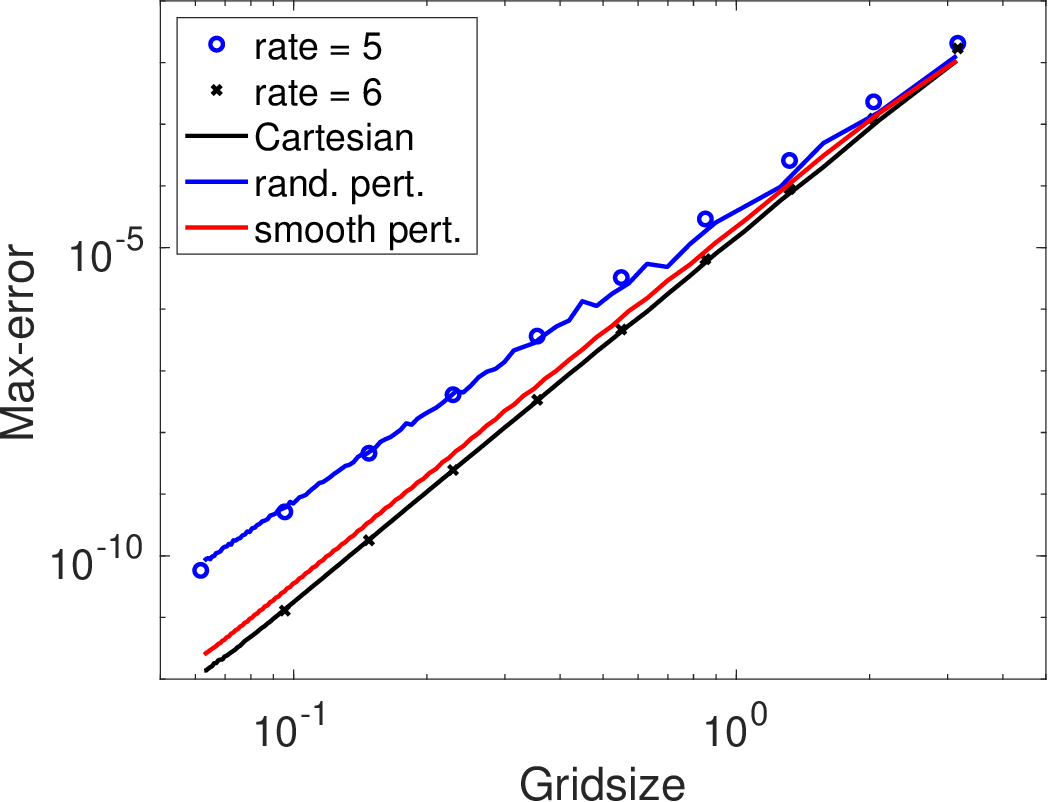}
\caption{Dirichlet problem in a square. Errors measured in the $l^2$ and the $l^\infty$ norms for the three different sequences of grids. The top row is for $k=2$ and the bottom row is for $k=4$. \label{fig:D_square_mode_errs}}
\end{center}
\end{figure}

We evolve the numerical solution until time 5 by the time stepping discussed in Section \ref{sec:timestepping} of order of accuracy 4 when $k=2$ and 6 when $k=4$. To get clean measurements of the error we report the time integrated errors 
\[
\left(\int_0^{5} \|u(\cdot,t) - u_h(\cdot,t)\|^2_{l^2} \, dt \right)^{\frac{1}{2}}, \ \ \ \ \int_0^{5} \|u(\cdot,t) - u_h(\cdot,t)\|_{l^\infty} \, dt,
\]
for the spatial $l^2$ and $l^\infty$ errors respectively.

The results are displayed in Figure \ref{fig:D_square_mode_errs}. Note that here and in the rest of this section the solid lines in the figures are the computed errors, using many different grid sizes, and the symbols are indicating the slopes or rates of convergence of the curves. 
 The Cartesian grids and smoothly perturbed grids satisfy the assumptions of the theory developed in this paper while the second sequence of randomly perturbed grids does not.
The results confirm the theoretical predictions  for smooth variable coefficients  as the rate of convergence is $k+2$ for the $l^2$-norm in the cases of the Cartesian meshes and the smoothly perturbed meshes. 
We also observe the rate $k+2$ in the $l^\infty$-norm for these cases. For the non-smooth variable coefficients resulting from the randomly perturbed grid, which is not covered by our theory, we see a rate of convergence of $k+1$ in the $l^2$-norm. 

\subsubsection{Standing mode in a sector of an annulus with Dirichlet conditions}  \label{sec:Dann_mode}
\begin{figure}[htb]
\begin{center}
\includegraphics[width=0.45\textwidth]{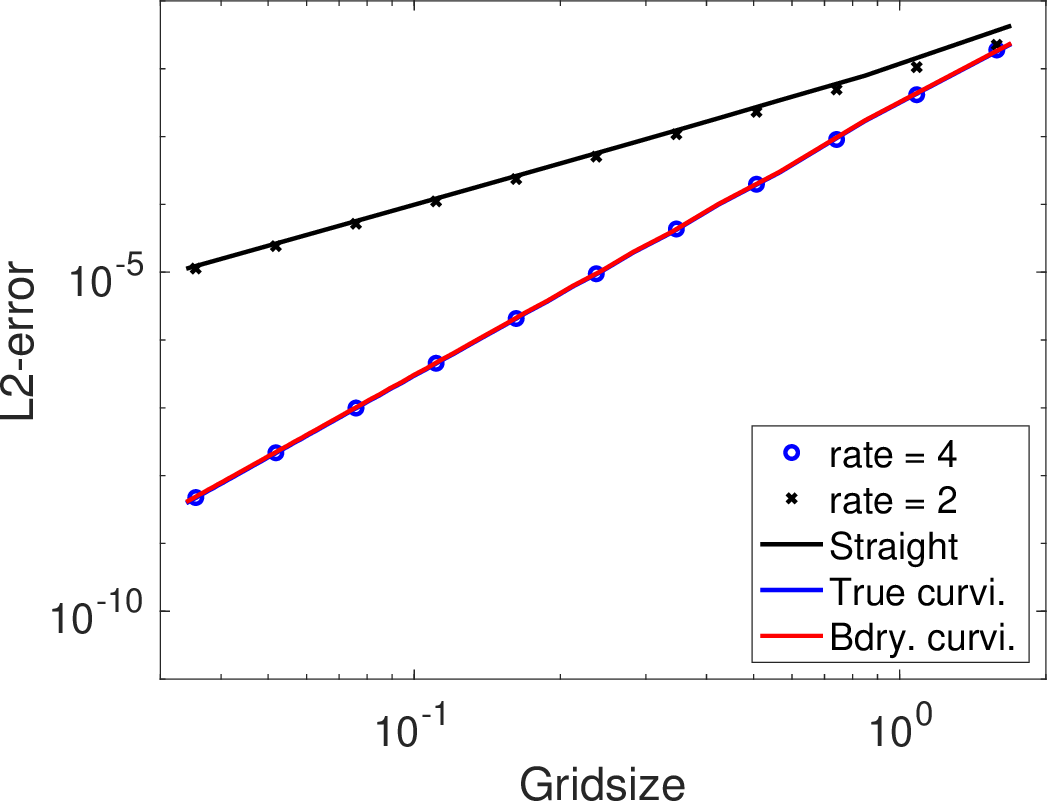}
\includegraphics[width=0.45\textwidth]{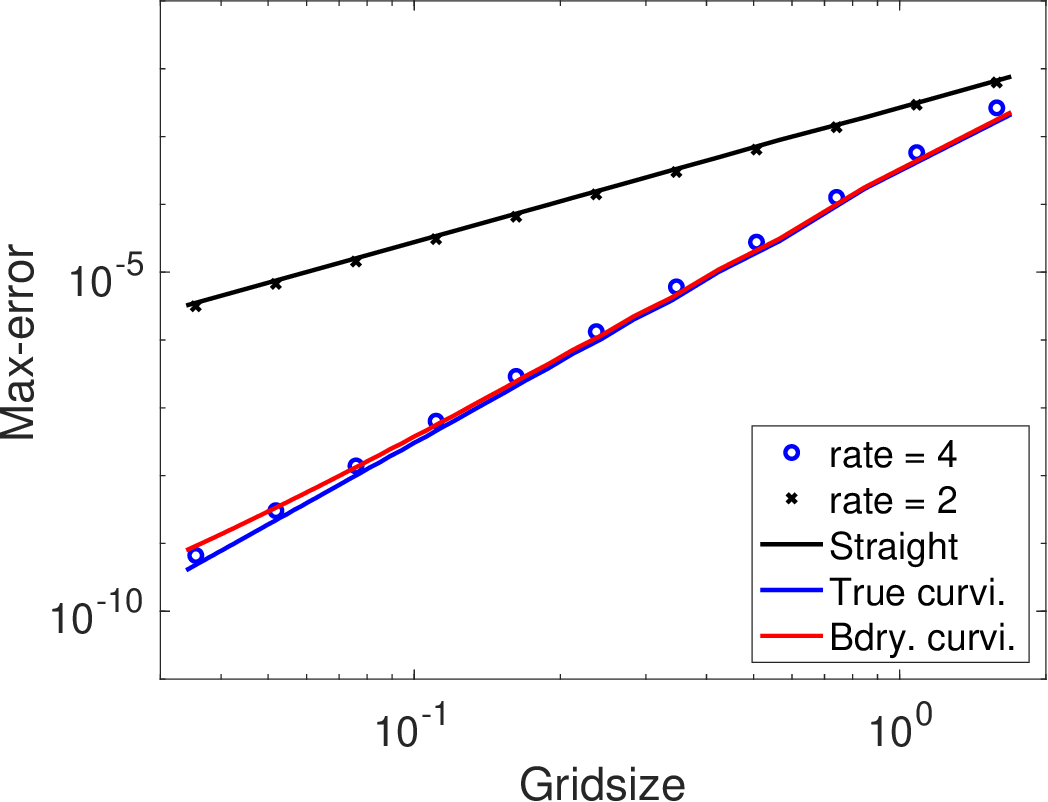}
\includegraphics[width=0.45\textwidth]{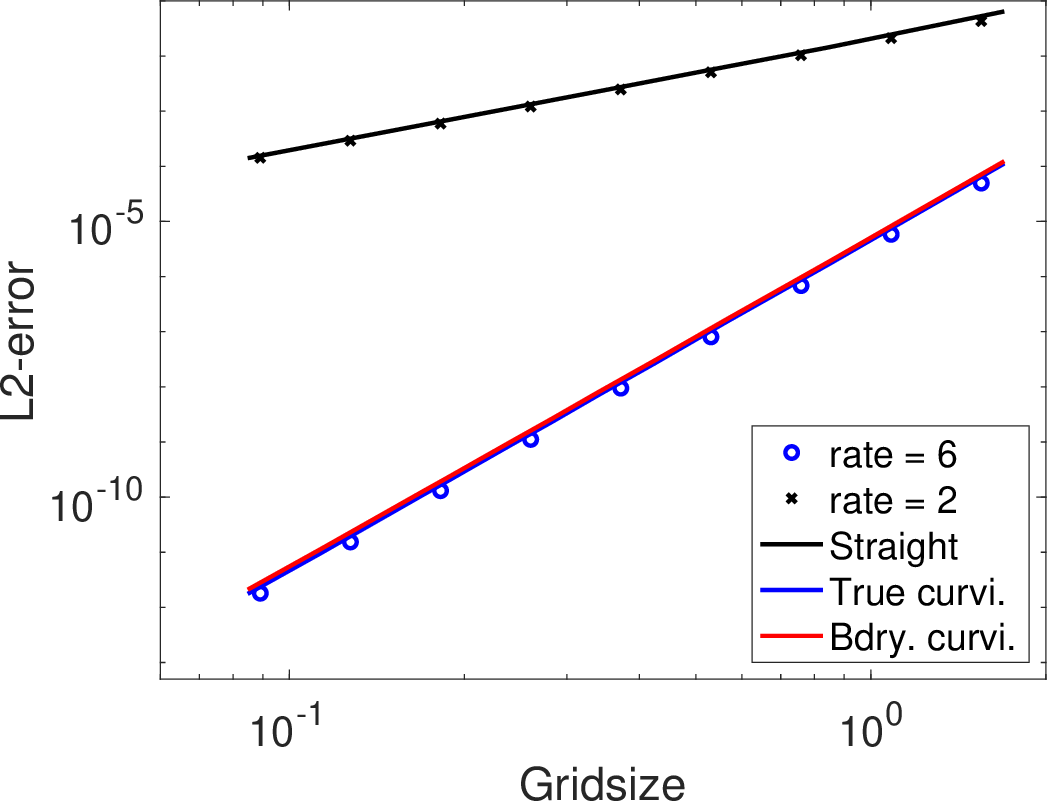}
\includegraphics[width=0.45\textwidth]{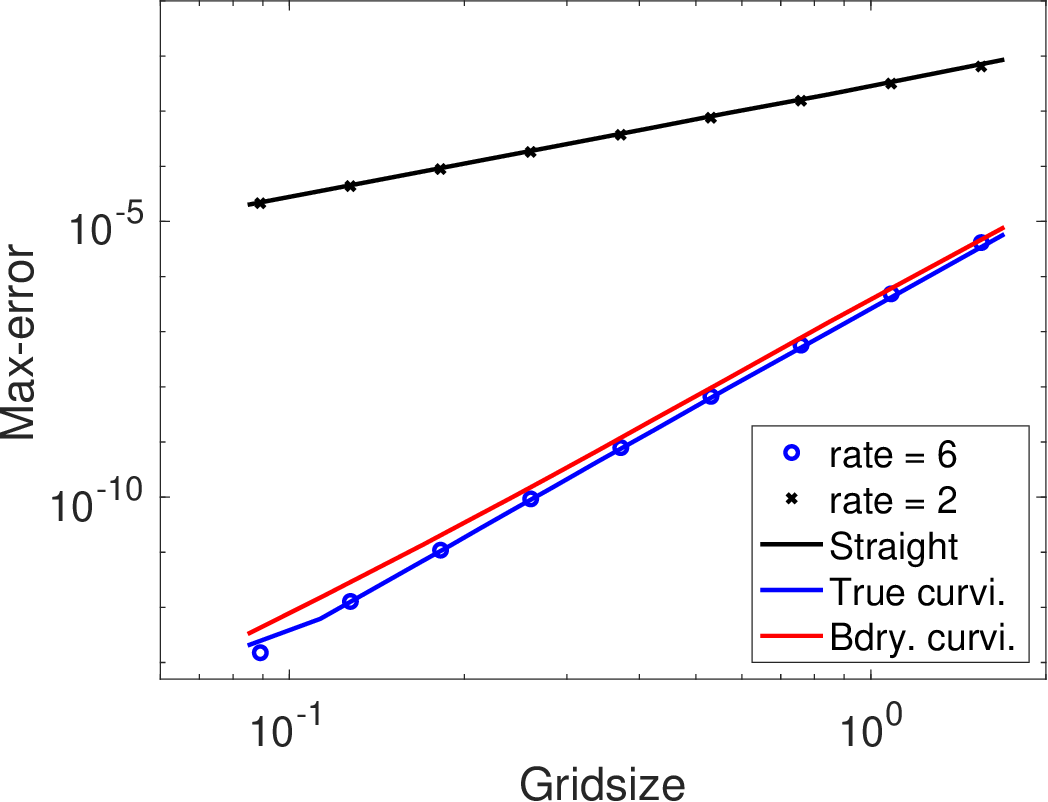}
\caption{Dirichlet problem in an annular sector. Errors measured in the $l^2$ and the $l^\infty$ norms for the three different sequences of grids. The top row is for $k=2$ and the bottom row is for $k=4$. These results are for the annular problem with homogenous Dirichlet boundary conditions. \label{fig:ann_mode_errs_dir}}
\end{center}
\end{figure}	

In this experiment we solve the wave equation  $u_{tt}=u_{xx}+u_{yy}$  with homogenous Dirichlet boundary conditions. The computational domain is the first quadrant of the annular region between two circles with radii $r_{\rm i}= 7.58834243450380438$ and  $r_{\rm o} = 14.37253667161758967$, i.e. the domain is described by $(x,y) = (r \cos \theta, r \sin \theta)$ where
\[
 r_{\rm i} \le r \le r_{\rm o}, \ \ 0 \le \theta \le \pi/2.
\]

On this domain the standing mode 
\[
u(r,\theta,t) = J_4(r) \sin(4\theta) \cos(t),
\] 
is an exact solution and we use this solution to specify the initial conditions and to compute errors.  

We consider the two cases $k = 2$ and $k=4$ and discretize on three different sequences of grids. The first sequence uses a straight sided approximation of the annulus  and all internal elements are quadrilaterals with straight sides. The second sequence uses curvilinear elements throughout the domain and all internal element boundaries conform with the polar coordinate transformation.  After the smooth mapping to the unit square, smooth variable coefficients emerge due to the geometric terms.  The metric terms are approximated with numerical differentiation using the values at the quadrature points. The third sequence is the same as the second sequence but all the internal element edges are straight. The meshes in the last sequence are likely close to those that would be provided by most grid generators. 

We evolve the numerical solution until time 1 by the time stepping discussed in Section \ref{sec:timestepping} of order of accuracy 4 when $k=2$ and 6 when $k=4$. Again, to get clean measurements of the error we report the time integrated errors 
\[
\left(\int_0^{1} \|u(\cdot,t) - u_h(\cdot,t)\|^2_{l^2} \, dt \right)^{\frac{1}{2}}, \ \ \ \ \int_0^{1} \|u(\cdot,t) - u_h(\cdot,t)\|_{l^\infty} \, dt,
\]
for the spatial $l^2$ and $l^\infty$ errors respectively. 

The results are displayed in Figure \ref{fig:ann_mode_errs_dir}. Here, as expected, we only observe second order accuracy independent of $k$ for the non-geometry-conforming meshes. We observe a convergence at the rate of $k+2$ in both the $l^2$-norm and $l^\infty$-norm for the geometry-conforming meshes.   
 The true curvilinear grids are covered by our theory since the variable coefficients due to the geometric transformation are smooth. For the third sequence of grids, since internal edges are straightsided,
the variable coefficients from the geometric transformation are not smooth across edges thus this configuration is not covered by our theory. Nonetheless, its convergence rate is still $k+2$.  

\subsubsection{Standing mode with Neumann conditions} \label{sec:square_mode}
\begin{figure}[htb]
\begin{center}
\includegraphics[width=0.45\textwidth]{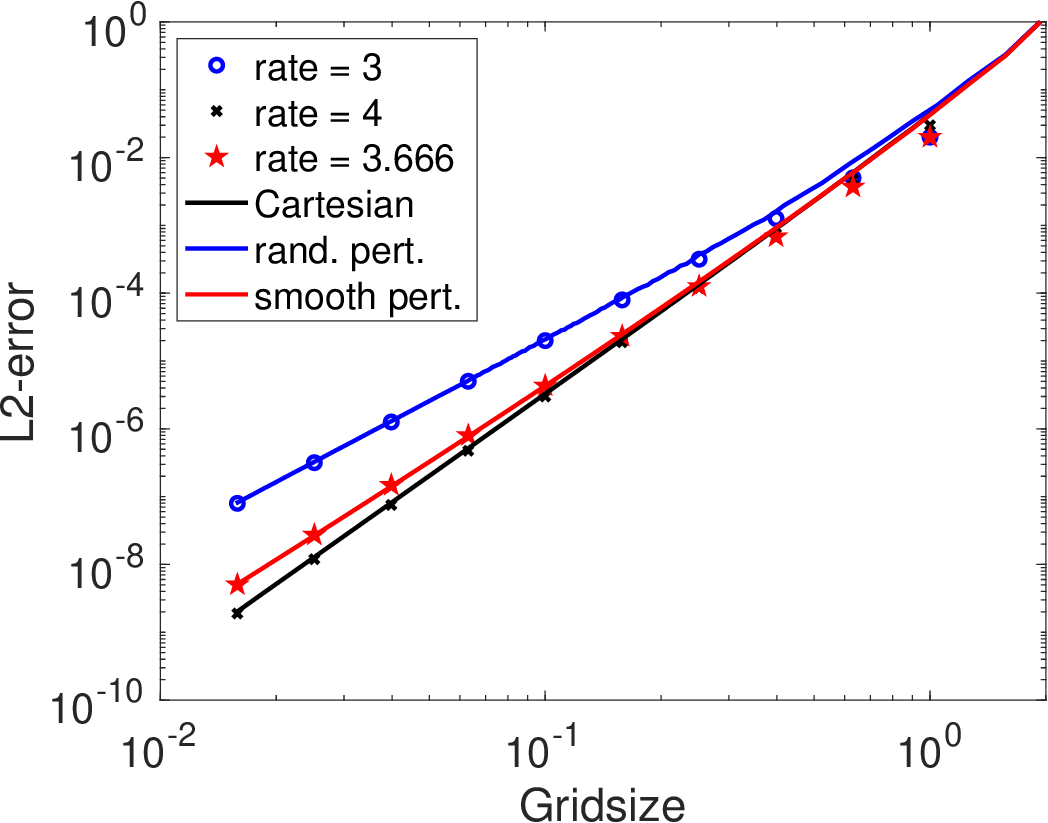}
\includegraphics[width=0.45\textwidth]{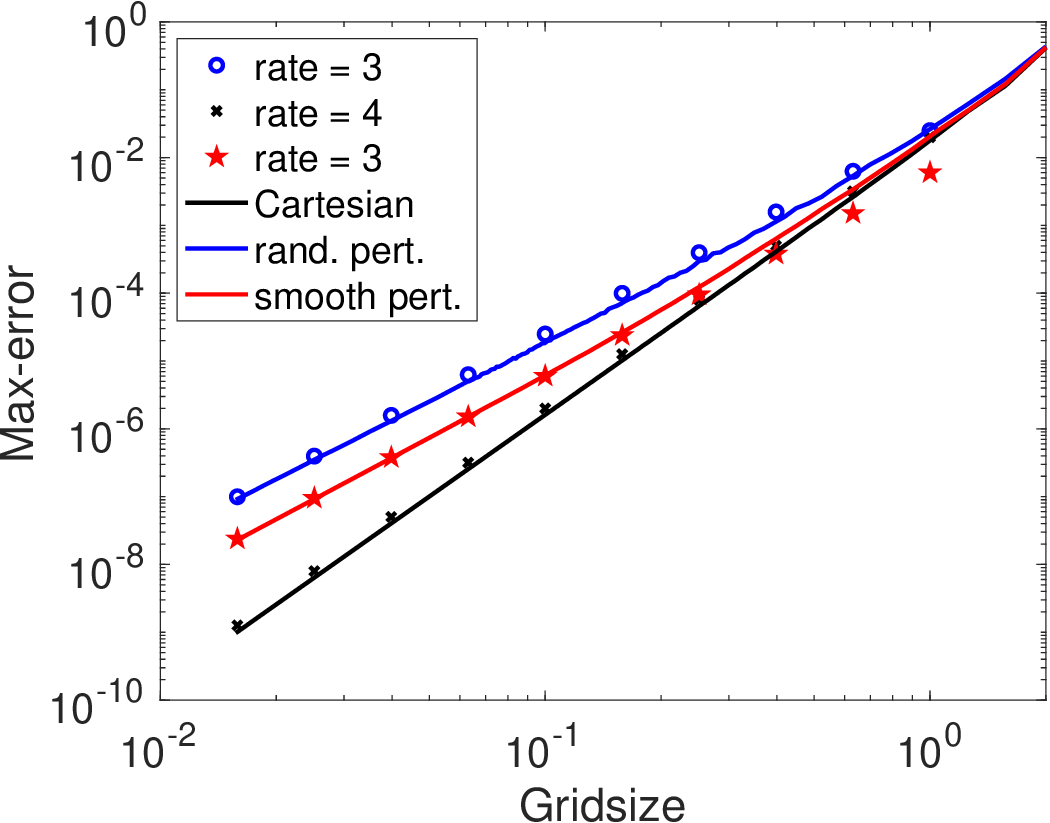}
\includegraphics[width=0.45\textwidth]{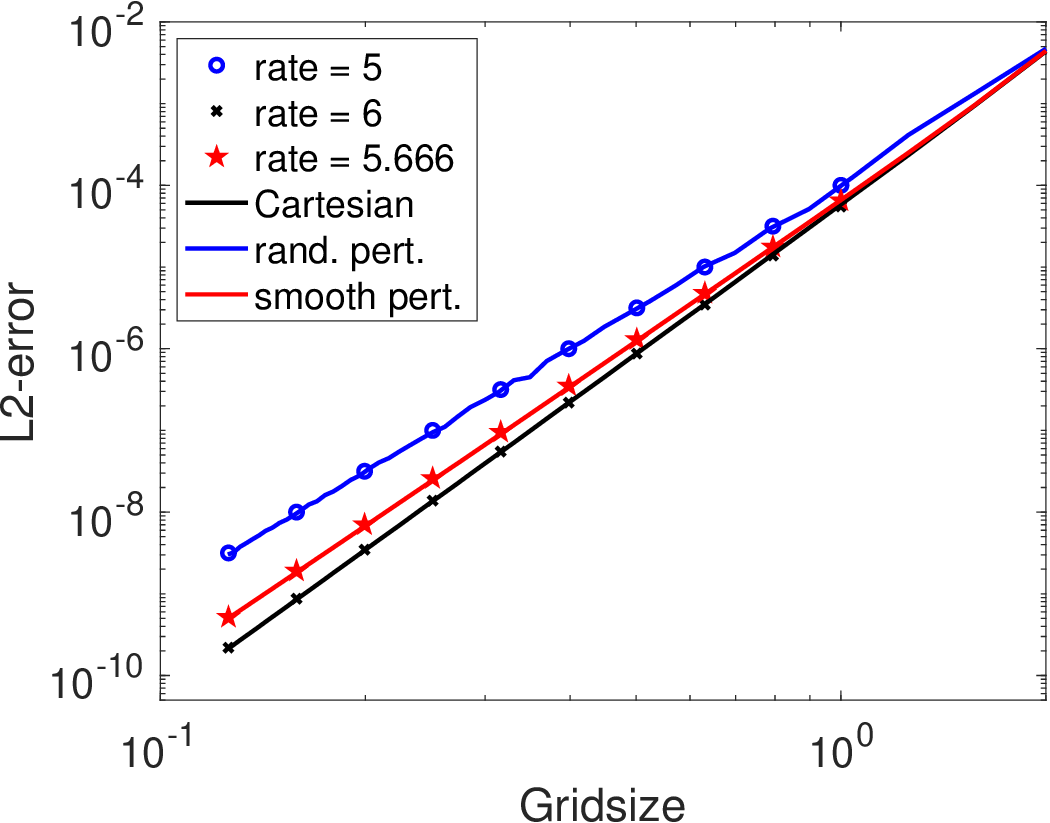}
\includegraphics[width=0.45\textwidth]{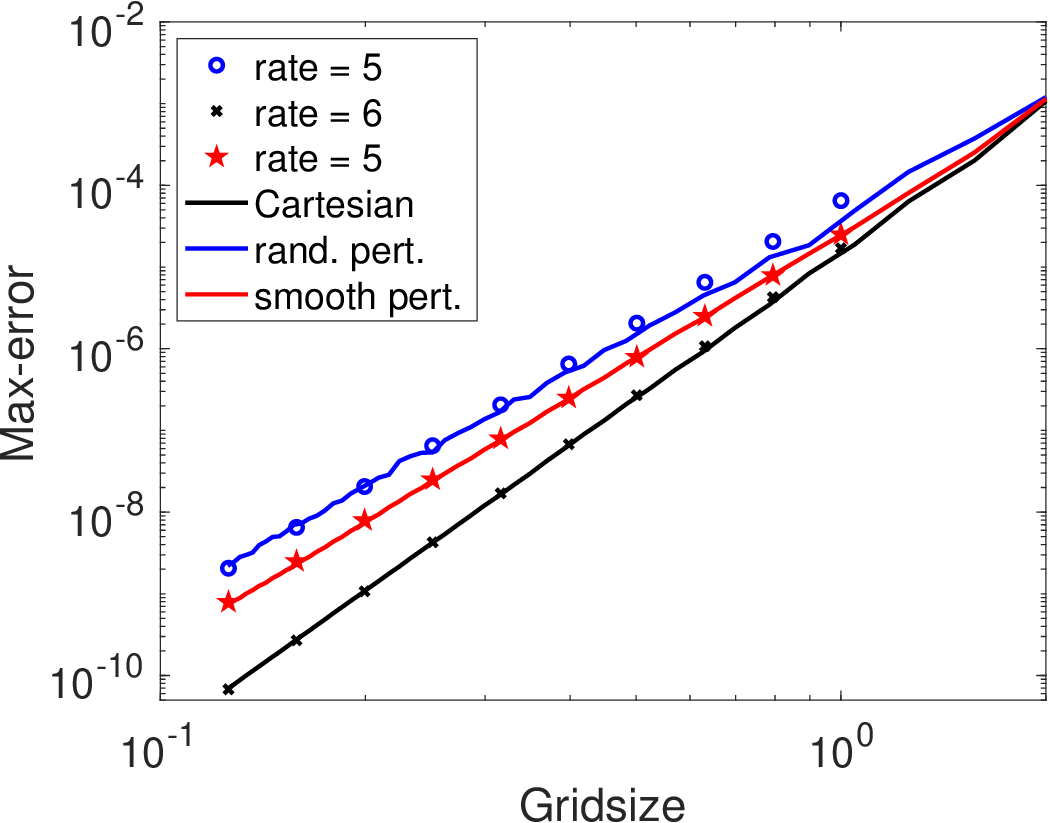}
\caption{Neumann square problem. Errors measured in the $l^2$ and the $l^\infty$ norms for the three different sequences of grids. The top row is for $k=2$ and the bottom row is for $k=4$. \label{fig:square_mode_errs}}
\end{center}
\end{figure}	
In this experiment we approximate the solution to the wave equation   $u_{tt}=u_{xx}+u_{yy}$   in the square domain $(x,y) \in [-\pi,\pi]^2$. Then with homogenous Neumann boundary conditions and initial data 
\[
u(x,y,0) = \cos (x) \cos(y), \ \ u_t(x,y,0) = 0,
\] 
the exact standing mode solution is
\[
u(x,y,0) = \cos (x) \cos(y) \cos(\sqrt{2}t).
\]

We consider the two cases $k = 2$ and $k=4$ and discretize on the same three sequences of grids as those used in \S \ref{sec:Dsquare_mode}. We evolve the numerical solution until time 5 as above and we report the time integrated errors as above. 

The results are displayed in Figure \ref{fig:square_mode_errs}. 
 For the Cartesian mesh  we observe a rate of convergence $k+2$ in the $\ell^2$-norm, confirming our theory. 
For the smoothly perturbed grids, which corresponds to   smooth variable coefficients resulting in mixed second order derivatives on the reference rectangular mesh,
  the rate in the $l^2$-norm appears to be $k+5/3$. As explained in Section \ref{sec-Neumann}, only $(k+\frac32)$-th order can be proven when both mixed second order derivatives and Neumann boundary conditions are involved.
 As in the Dirichlet case, the randomly perturbed grid yields rates of convergence   $k+1$ in both norms. 
\subsubsection{Standing mode in a sector of an annulus with Neumann conditions} \label{sec:ann_mode}
In this experiment we solve the the wave equation  $u_{tt}=u_{xx}+u_{yy}$  with homogenous Neumann boundary conditions. The computational domain is again the first quadrant of the annular region between two circles, now with radii $r_{\rm i}= 5.31755312608399$ and  $r_{\rm o} = 9.28239628524161$, to satisfy the boundary conditions.
On this domain the standing mode 
\[
u(r,\theta,t) = J_4(r) \cos(4\theta) \cos(t),
\] 
is an exact solution and we use this solution to specify the initial conditions and to compute errors.  

As in the previous examples we consider the two cases $k = 2$ and $k=4$ and discretize on the same three different sequences of grids as was used in the Dirichlet example above. We evolve the numerical solution until time 1 in the same way as above and we report the time integrated errors.

The results are displayed in Figure \ref{fig:ann_mode_errs}. 
 Here, the only grid satisfying our assumptions is the true curvilinear grid. For this case, the problem  is equivalent to solving a variable coefficient problem $u_{tt}=u_{rr}+\frac{1}{r^2}u_{\theta\theta}+\frac{1}{r} u_r$ on rectangular meshes for polar coordinates $(r,\theta)\in [r_i, r_o]\times [0,\frac{\pi}{2}]$. Since there are no mixed second order derivatives,
by our theory as explained in Section   \ref{sec-Neumann}, $(k+2)$-th order in the $\ell^2$-norm can still be proven. We can see that the rate for the true curvilinear grid is indeed 
$k+2$ in $\ell^2$-norm, confirming our theory for Neumann boundary conditions.

\begin{figure}[htb]
\begin{center}
\includegraphics[width=0.45\textwidth]{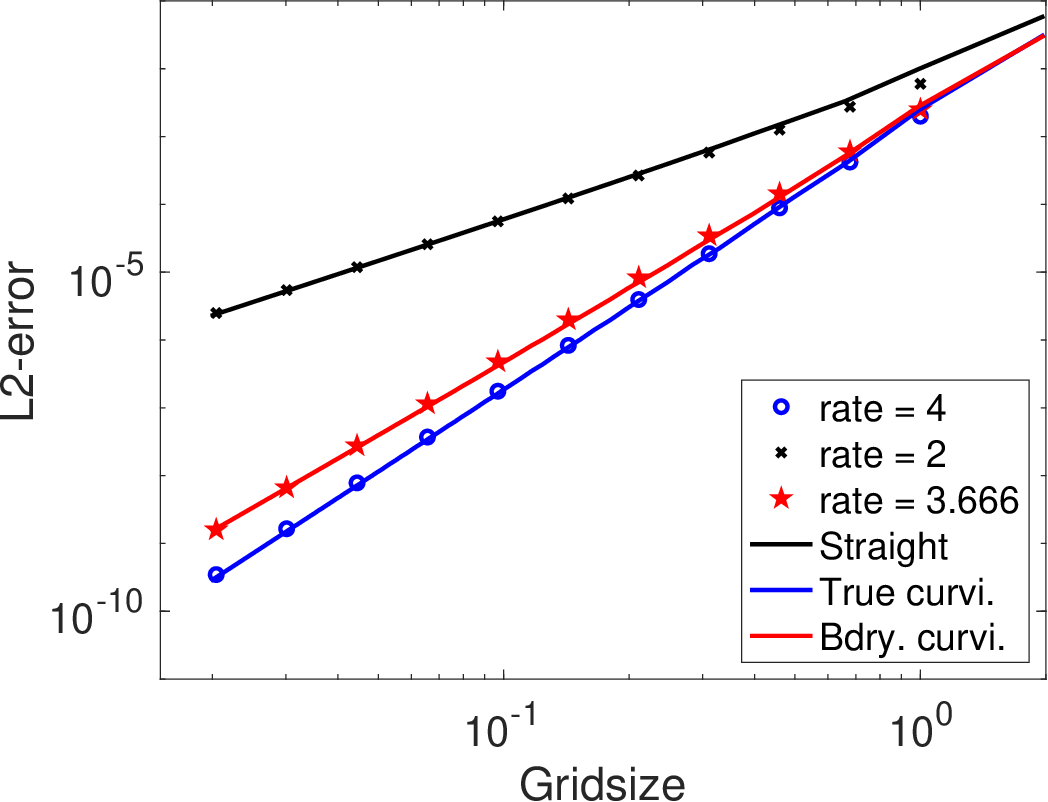}
\includegraphics[width=0.45\textwidth]{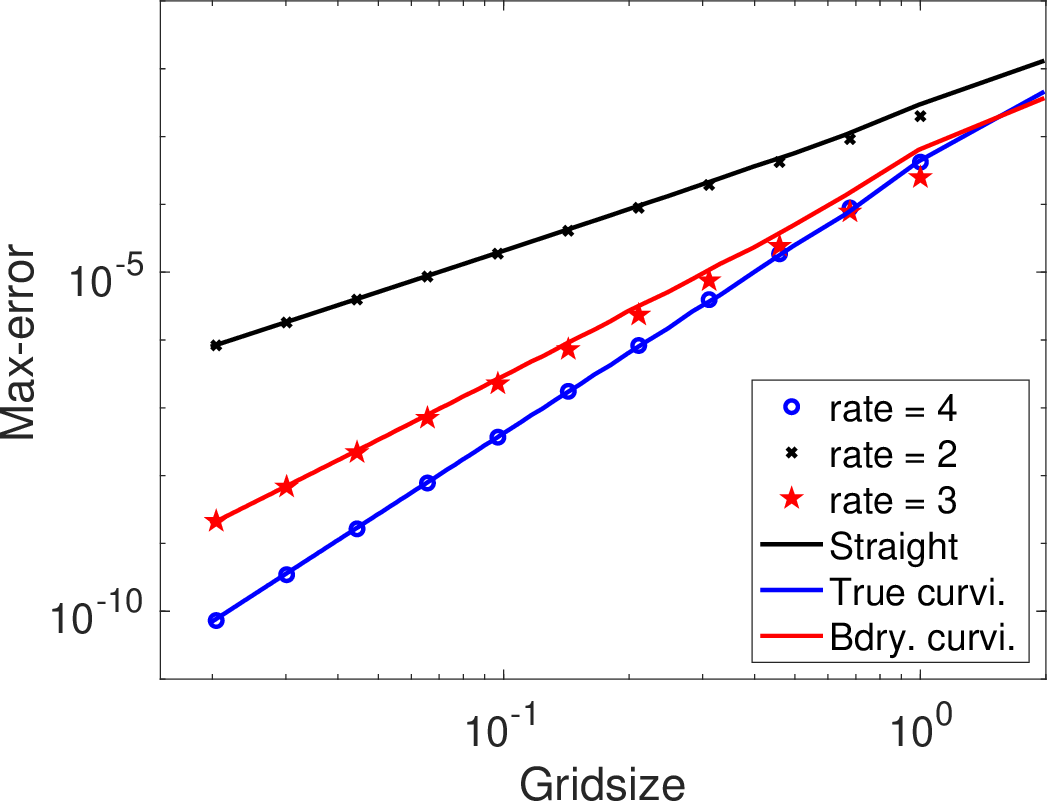}
\includegraphics[width=0.45\textwidth]{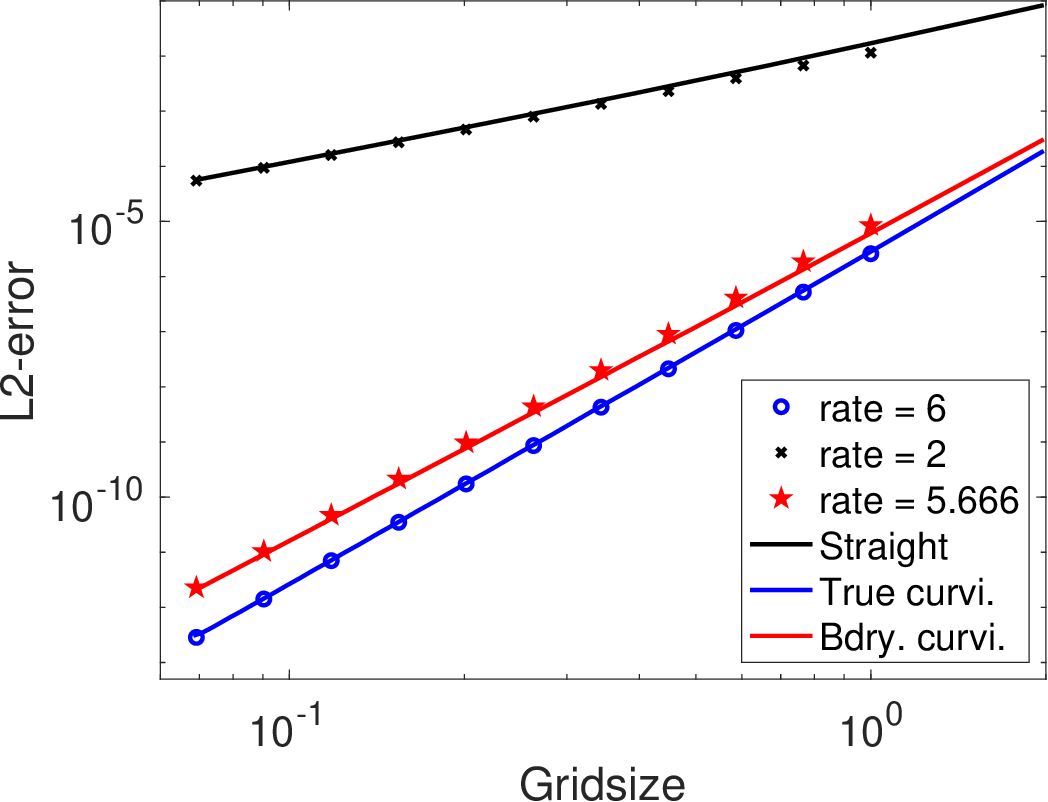}
\includegraphics[width=0.45\textwidth]{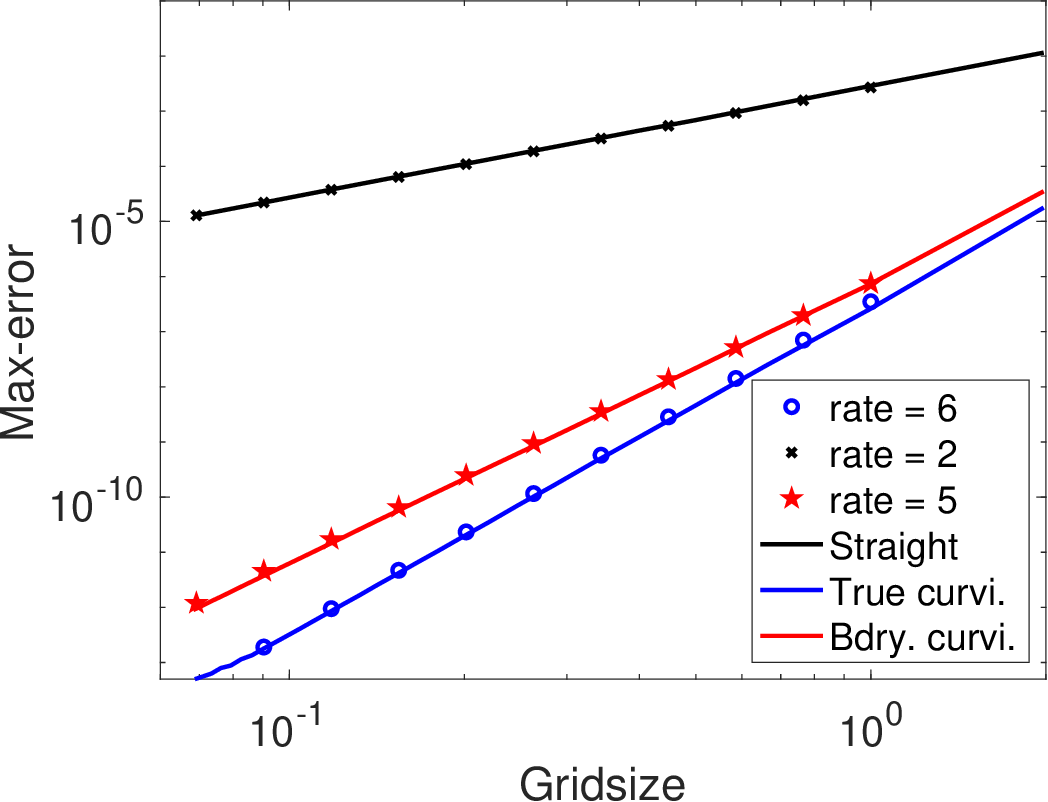}
\caption{Neumann annular sector problem. Errors measured in the $l^2$ and the $l^\infty$ norms for the three different sequences of grids. The top row is for $k=2$ and the bottom row is for $k=4$. These results are for the annular problem with homogenous Neumann conditions. \label{fig:ann_mode_errs}}
\end{center}
\end{figure}

\subsection{Numerical tests for the parabolic equation}
For problem \eqref{heat-equation-homo-bc} on the domain $\Omega = (0,\pi)^2$, we set $\mathbf a=\left( {\begin{array}{cc}
   a_{11} & a_{12} \\
   a_{21} & a_{22} \\
  \end{array} } \right)$ with 
\begin{equation*}
\begin{aligned}  
  a_{11} =& \left(\frac34 + \frac14 \sin(t)\right)\left(1+y+y^2+x\cos{y}\right),\\
  a_{12} =& a_{21}=\left(\frac34 + \frac14 \sin(t)\right)\left(1+\frac12(\sin(\pi x)+x^3)(\sin(\pi y)+y^3)+\cos(x^4+y^3)\right),\\
  a_{22} =&\left(\frac34 + \frac14 \sin(t)\right)\left(1+x^2\right),
  \end{aligned} 
  \end{equation*}
   $\mathbf b=\left( {\begin{array}{cc}
   b_{1} \\
   b_{2} \\
  \end{array} } \right)$ with 
  \begin{equation*}
\begin{aligned}  
b_{1}=\left(\frac34 + \frac14 \sin(t)\right)\left(\frac15+x\right),  
b_{2}=\left(\frac34 + \frac14 \sin(t)\right)\left(\frac15-y\right), 
    \end{aligned} 
  \end{equation*}
and $c = \left(\frac34 + \frac14 \sin(t)\right)\left(10+x^4y^3\right)$. For time discretization in \eqref{heat-scheme}, we use the third order backward differentiation formula (BDF) method. 
  Let $u(x,y,t) = (\frac34 + \frac14 \sin(t))(-\sin(y)\cos(y)\sin(x)^2)$ and we use a potential function $f$ so that $u$ is the exact solution. The time step is set as $\Delta t = \min(\frac{\Delta x}{10}, \frac{\Delta x}{10b_M}, \frac{f_M}{10})$, where $b_M = \max_{\mathbf{x}\in \Omega, i = 1,2}|b_{i}(0,\mathbf{x})|$ and $f_M = \max_{\mathbf{x}\in \Omega}|f(0,\mathbf{x})|$. The errors at time $T = 0.1$ are listed in Table \ref{table-heat}, in which we observe order around $k+2$ for the $\ell^2$-norm.
  
\begin{table}[h]
\centering
\caption{A two-dimensional parabolic equation with Dirichlet boundary conditions.}
\begin{tabular}{|c |c |c c|c c|}
\hline $Q^k$ polynomial  &  SEM Mesh & $l^2$ error  &  order & $l^\infty$ error & order \\
\hline
\multirow{4}{*}{k = 2} 
  &  $4\times 4$ & 8.34E-3 & - & 4.57E-3 & -\\\cline{2-6}
  &  $8\times 8$ & 6.59E-4 & 3.66 & 3.16E-4 & 3.85\\\cline{2-6}
  &  $16\times 16$ & 4.52E-5 & 3.86 & 2.36E-5 & 3.74 \\ \cline{2-6}
  &  $32\times 32$ & 2.91E-6 & 3.96 & 1.53E-6 & 3.94\\ 
\hline
\multirow{4}{*}{k = 3} 
  &  $4\times 4$ & 5.88E-4 & - & 1.71E-4 & -\\\cline{2-6}
  &  $8\times 8$ & 2.24E-5 & 4.71 & 7.56E-6 & 4.50\\\cline{2-6}
  &  $16\times 16$ & 7.49E-7 & 4.90 & 2.52E-7 & 4.91 \\ \cline{2-6}
  &  $32\times 32$ & 2.38E-8 & 4.97 & 8.06E-9 & 4.96\\ 
\hline
\multirow{4}{*}{k = 4} 
  &  $4\times 4$ & 4.26E-5 & - & 1.16E-5 & -\\\cline{2-6}
  &  $8\times 8$ & 7.62E-7 & 5.81 & 2.34E-7 & 5.63\\\cline{2-6}
  &  $16\times 16$ & 1.26E-8 & 5.92 & 4.12E-9 & 5.83 \\ \cline{2-6}
  &  $32\times 32$ & 2.00E-10 & 5.98 & 6.68E-11 & 5.95\\ 
  \hline
\end{tabular}
\label{table-heat}
\end{table}

\subsection{Numerical tests for the linear Schr\"{o}dinger equation}
For problem \eqref{schrodinger-eqn} on the domain $(0,2)^2$,  a fourth-order explicit Adams-Bashforth as time discretization for \eqref{schrodinger-scheme}. The solution and potential functions are as follows: $u(x,y,t) = e^{-it}e^{-\frac{x^2+y^2}{2}}$, $V(x,y)=\frac{x^2+y^2}{2}$, and $f(x,y,t)=0$. The time step is set as $\Delta t = \frac{\Delta x^2}{500}$. 
Errors at time $T = 0.5$ are listed in Table \ref{table-schrodinger}, in which we observe order near $k+2$ for the $\ell^2$-norm.

\begin{table}[h]
\centering
\caption{A two-dimensional linear Schr\"{o}dinger equation with Dirichlet boundary conditions.}
\begin{tabular}{|c |c |c c|c c|}
\hline $Q^k$ polynomial  &  SEM Mesh & $l^2$ error  &  order & $l^\infty$ error & order \\
\hline
\multirow{4}{*}{k = 2} 
  &  $4\times 4$ & 9.98E-4 & - & 6.36E-4 & -\\\cline{2-6}
  &  $8\times 8$ & 6.65E-5 & 3.91 & 4.01E-5 & 3.99\\\cline{2-6}
  &  $16\times 16$ & 4.10E-6 & 4.02 & 2.77E-6 & 3.85 \\ \cline{2-6}
  &  $32\times 32$ & 2.53E-7 & 4.02 & 1.79E-7 & 3.89\\ 
\hline
\multirow{4}{*}{k = 3} 
  &  $4\times 4$ & 4.06E-5 & - & 2.12E-5 & -\\\cline{2-6}
  &  $8\times 8$ & 1.12E-6 & 5.18 & 5.56E-7 & 5.26\\\cline{2-6}
  &  $16\times 16$ & 3.22E-8 & 5.12 & 1.75E-8 & 4.99 \\ \cline{2-6}
  &  $32\times 32$ & 1.05E-9 & 4.94 & 5.33E-10 & 5.04\\ 
\hline
\multirow{4}{*}{k = 4} 
  &  $4\times 4$ & 1.61E-6 & - & 5.86E-7 & -\\\cline{2-6}
  &  $8\times 8$ & 2.65E-8 & 5.92 & 9.93E-9 & 5.88\\\cline{2-6}
  &  $16\times 16$ & 3.95E-10 & 6.07 & 1.66E-10 & 5.90 \\ \cline{2-6}
  &  $32\times 32$ & 5.30E-12 & 6.22 & 2.66E-12 & 5.97\\ 
\hline
\end{tabular}
\label{table-schrodinger}
\end{table}

\section{ Concluding remarks}
\label{concludingremark}
 
We have proven that the $Q^k$ ($k\geq 2$) spectral element method, when regarded as a finite difference scheme, is a $(k+2)$-th order accurate scheme 
in the discrete 2-norm for  linear hyperbolic, parabolic and Schr\"{o}dinger equations with Dirichlet boundary conditions, under smoothness assumptions of the exact solution 
and the differential operator coefficients.  The same result holds for Neumann boundary conditions when there are no mixed second order derivatives. 
This explains the observed order of accuracy when the errors of the spectral element method are only measured at nodes of degree of freedoms.  
\bibliographystyle{siamplain}
\bibliography{ref,appelo,hag} 

\end{document}